\documentclass{amsart}
\usepackage{amssymb,amsmath, amsthm,latexsym}
\usepackage{psfrag}
\usepackage{amscd}
\usepackage{graphicx}
\newcommand{\cal}[1]{\mathcal{#1}}
\theoremstyle{plain}
\newtheorem{theo}{Theorem}

\newtheorem{lemma}{Lemma}[section]
\newtheorem{theorem}[lemma]{Theorem}
\newtheorem{proposition}[lemma]{Proposition}
\newtheorem{corollary}[lemma]{Corollary}
\theoremstyle{definition}
\newtheorem{definition}[lemma]{Definition}
\newtheorem{remark}[lemma]{Remark}
\parskip=\bigskipamount

\let\egthree=\phi
\let\phi=\varphi
\let\varphi=\egthree

\begin{document}
\title{Symbolic dynamics for the Teichm\"uller flow}
\author{Ursula Hamenst\"adt}
\thanks
{
Research
partially supported by the Hausdorff Center Bonn}
\date{April 12, 2025}

\begin{abstract}
Let 
${\cal Q}$ be a component of
a stratum of abelian or quadratic differentials on an 
oriented surface of
genus $g\geq 0$ with $m\geq 0$ punctures
and $3g-3+m\geq 2$. 
We construct a subshift of finite type 
$(\Omega,\sigma)$ 
and a Borel suspension 
of $(\Omega,\sigma)$ 
which admits a finite-to-one
semi-conjugacy into the Teichm\"uller flow $\Phi^t$ on 
${\cal Q}$. This is used to show that  
the $\Phi^t$-invariant Lebesgue measure $\lambda$ on ${\cal Q}$ 
is the unique measure of 
maximal entropy.
\end{abstract}

\maketitle


\section{Introduction}

A surface $S$ \emph{of finite type} is a
closed oriented surface of genus $g\geq 0$ with $m\geq 0$ marked
points, so-called \emph{punctures}.
We assume that $3g-3+m\geq 2$,
that is, $S$ is not a sphere with at most $4$
punctures or a torus with at most $1$ puncture.
We then call the surface $S$ \emph{nonexceptional}.
The Euler characteristic of $S$ is negative.

The \emph{Teichm\"uller space} ${\cal T}(S)$
of $S$ is the quotient of the space of all complete 
finite area hyperbolic
metrics on the complement of the punctures in $S$ under the action of the
group of diffeomorphisms of $S$ which are isotopic
to the identity. The sphere bundle 
\[\tilde {\cal Q}(S)\to {\cal T}(S)\] 
of all \emph{holomorphic
quadratic differentials} of area
one can naturally be identified with the unit cotangent
bundle for the \emph{Teichm\"uller metric}.
If the surface $S$ has punctures, that is, if $m>0$, then
we define a holomorphic quadratic differential 
on $S$ to be a meromorphic quadratic differential on 
the closed Riemann surface obtained from $S$ by filling
in the punctures, with a simple pole at each
of the punctures and no other poles.

The 
\emph{mapping class group}
${\rm Mod}(S)$ of all isotopy classes of
orientation preserving diffeomorphisms of $S$
naturally acts on $\tilde {\cal Q}(S)$. 
The quotient 
\[{\cal Q}(S)=\tilde {\cal Q}(S)/{\rm Mod}(S)\]
is the \emph{moduli space of area one
quadratic differentials}.
It can be partitioned into
so-called \emph{strata}. Namely, let
$1\leq m_1\leq \dots \leq m_\ell$ $(\ell\geq 1)$ 
be a sequence of positive integers 
with \[\sum_im_i=4g-4+m.\] 
The stratum 
${\cal Q}(m_1,\dots,m_\ell;-m)$ defined by
the $\ell$-tuple $(m_1,\dots,m_\ell)$ 
is the moduli space of pairs $(C,\phi)$ 
where $C$ is a closed Riemann surface
of genus $g$ and where 
$\phi$ is an area one meromorphic quadratic differential on $C$ with
$\ell$ zeros of order $m_i$ and $m$ simple poles
and which is not the square of a holomorphic one-form.

A stratum 
${\cal Q}(m_1,\dots,m_\ell;-m)$ is a real
hypersurface in a complex orbifold of 
complex dimension 
\[h=2g+\ell+m-2.\] 
Strata need not be
connected, however they have at most two connected
components \cite{L08}. 
The closure in ${\cal Q}(S)$ 
of a component of a 
stratum ${\cal Q}(m_1,\dots,m_\ell;-m)$ 
is a union of components of 
strata ${\cal Q}(n_1,\dots,n_s;-m^\prime)$ where $s\leq \ell,m^\prime \leq m$.
Note here that it is natural to allow a simple pole to merge with a zero in
the closure, thus decreasing the number $m$ of marked points.

If the surface $S$ is closed, that is, if $m=0$, 
then we can also consider the 
bundle 
\[\tilde {\cal H}(S)\to {\cal T}(S)\] of
area one \emph{abelian differentials}.
It descends to the moduli
space ${\cal H}(S)$ 
of holomorphic one-forms defining a singular
euclidean metric of area one.
Again this moduli space decomposes into 
a union of strata ${\cal H}(k_1,\dots,k_s)$ corresponding
to the orders of the zeros of the differentials.
Strata are in general not connected, but there are at most three
connected components 
\cite{KZ03}. The stratum 
${\cal H}(k_1,\dots,k_s)$
is a real hypersurface in a complex
orbifold of dimension
\[h=2g+s-1.\]

The \emph{Teichm\"uller flow} $\Phi^t$ 
acts on ${\cal Q}(S)$ (or ${\cal H}(S)$) preserving
the strata. If ${\cal Q}$ is a component of a 
stratum of abelian differentials then 
\emph{Rauzy induction} for interval exchange
transformations can be used to construct 
a \emph{symbolic coding} for the 
Teichm\"uller flow on ${\cal Q}$
(\cite{V82}, see also \cite{AGY06} for a discussion and
references). Rauzy induction has been extended to strata
of quadratic differentials by Boissy and Lanneau \cite{BL07}. It was used
to identify connected components of strata of quadratic differentials.

Our main goal is to construct a new coding for 
the Teichm\"uller flow on any component ${\cal Q}$ of a stratum. This coding
is based on 
the perspective on components of strata of abelian or quadratic differentials
developed in \cite{H22}, and it 
is well suited to study the space ${\cal M}_{\rm inv}({\cal Q})$
of all $\Phi^t$-invariant Borel probability
measures on ${\cal Q}$, equipped with the weak$^*$-topology. 

For the formulation of our main result, recall that a \emph{biinfinite subshift of finite type}
$(\Omega,\sigma)$ is defined by a finite alphabet ${\cal A}=\{1,\dots,p\}$ 
and a $(p,p)$-matrix $(a_{ij})$ with entries in $\{0,1\}$ such that
\[\Omega=\{(x_i)\in {\cal A}^{\mathbb{Z}}\mid a_{x_ix_{i+1}}=1 \text{ for all }i\}.\]
The \emph{shift} $\sigma$ acts on $\Omega$ by $\sigma(x_i)=(x_{i+1})$. 
The set $\Omega$ carries a natural $\sigma$-invariant topology, and for this
topology, $\Omega$ is compact. 

The shift is called \emph{topologically transitive} if the $\sigma$-action has a 
dense orbit. A sequence $(x_i)\subset \Omega$ is called 
\emph{normal} if every finite string
$(y_i)_{1\leq i\leq k}$ with $a_{y_iy_{i+1}}=1$ for all 
$0\leq i\leq k-1$ occurs infinitely often in forward and backward direction as a substring of 
$(x_i)$. 

In the statement of the following result, spaces of 
probability measures are equipped with the weak$^*$-topology.

\begin{theo}\label{coding}
Let ${\cal Q}$ be a component of a stratum of quadratic or abelian
differentials.
Then there exists
\begin{itemize}
\item  a topologically transitive subshift of finite type 
$(\Omega,\sigma)$, 
\item a $\sigma$- invariant dense Borel set ${\cal U}\subset \Omega$
containing all normal sequences, 
\item 
a suspension $(X,\Theta^t)$ of $\sigma$ over ${\cal U}$, given
by a positive bounded continuous roof function on ${\cal U}$ 
\end{itemize}
and a finite-to-one semi-conjugacy 
$\Xi:(X,\Theta^t)\to ({\cal Q},\Phi^t)$
which maps the space of 
$\Theta^t$-invariant Borel probability measures on 
$X$ continuously onto ${\cal M}_{\rm inv}({\cal Q})$.
\end{theo}

As is the case for Rauzy induction or, more precisely, the zippered rectangle 
flow considered in \cite{V86}, the coding constructed in 
Theorem \ref{coding} can be thought of as a finite cover of the Teichm\"uller flow
on ${\cal Q}$. In particular, it maps the collection of all periodic orbits for $\sigma$ contained 
in ${\cal U}$ onto
the collection of all periodic orbits in ${\cal Q}$. However, a periodic orbit 
in ${\cal Q}$ may have more than one preimage in the suspension flow
$(X,\Theta^t)$, and the restriction of $\Xi$ to 
any such preimage may be a nontrivial finite covering of the periodic orbit.

Our construction is valid for strata of abelian differentials,
but it is different from Rauzy induction. A dictionary between
these two codings has yet to be established.

A specific example of a $\Phi^t$-invariant Borel probability 
measure on a component 
${\cal Q}$ of a stratum in the
Lebesgue measure class was constructed by  
Masur and Veech \cite{M82,V86}.
This measure $\lambda$ 
is ergodic \cite{M82,V86} and of full support, and its 
\emph{entropy} $h_\lambda$
coincides with the complex dimension $2g+\ell+m-2$ 
(or $2g+s-1$) of the
complex orbifold defining the stratum
(note that we use a normalization for the Teichm\"uller
flow which is different from the one used by Masur and Veech). 
In particular,
the entropy of the Lebesgue measure on the open
connected stratum ${\cal Q}(1,\dots,1;-m)$ equals $6g-6+2m$.

Denote by $h_\nu$ the entropy of a 
measure $\nu\in {\cal M}_{\rm inv}({\cal Q})$.
Define
\[h_{\rm top}({\cal Q})=\sup\{h_\nu\mid 
\nu\in {\cal M}_{\rm inv}({\cal Q})\}.\]
A \emph{measure of maximal entropy} for the component
${\cal Q}$ is
a measure $\mu\in {\cal M}_{\rm inv}({\cal Q})$
such that $h_\mu=h_{\rm top}({\cal Q})$. 
Since ${\cal Q}$ is non-compact, 
a priori such a measure
need not exist. However, using Rauzy induction and the work 
of Buzzi and Sarig \cite{BS03}, 
Bufetov and Gurevich \cite{BG07} 
showed that for components of strata 
of abelian differentials, 
the $\Phi^t$-invariant probability measure
in the 
Lebesgue measure class is the unique
measure of
maximal entropy for the component.
We use Theorem \ref{coding} and \cite{BS03} to extend 
this result to all components of strata of 
quadratic or abelian differentials,
with a different proof.

\begin{theo}\label{entropymax} 
For every component ${\cal Q}$ of a stratum 
in ${\cal Q}(S)$ or ${\cal H}(S)$, 
the $\Phi^t$-invariant Borel probability measure
in the Lebesgue measure class is the unique measure
of maximal entropy.
\end{theo}

In view of the groundbreaking work of Eskin and Mirzakhani \cite{EM18} and of
Eskin, Mirzakhani and Mohammadi \cite{EMM15}, we expect that the analog of 
Theorem \ref{entropymax} also holds for arbitrary affine invariant 
manifolds in ${\cal Q}(S)$. However our methods do not apply in this
generality. Instead they are very well suited to study
the dynamics of the Teichm\"uller flow near the \emph{principal
boundary} of a stratum as initiated in \cite{H22}. This analysis is 
will be made precise in a sequel to this article.

\paragraph{\bf Organization of the article and outline of the proofs}
The organization of the article is as follows.
In Section \ref{sec:strata} we collect some results from 
\cite{H22} relating \emph{train tracks} to components of strata of 
abelian or quadratic differentials. 
This is used in Section \ref{sec:asymbolic}
to construct for every connected
component ${\cal Q}$ of a stratum
an associated topologically transitive subshift of finite
type $(\Omega,\sigma)$.

In Section \ref{sec:symbolicdyn} we define a \emph{roof 
function} $\rho$ on a $\sigma$-invariant 
dense Borel subset ${\cal U}$ of $\Omega$ containing all normal sequences and
use this roof function to define a suspension flow
$(X,\Theta^t)$ over ${\cal U}$. It fairly immediately follows from the
construction that 
there is a (partial) semi-conjugacy $\Xi:(X,\Theta^t)\to ({\cal Q},\Phi^t)$, that is,
the map $\Xi$ is continuous and commutes with the action of $\Theta^t$ on $X$ and 
$\Phi^t$ on ${\cal Q}$, however it is not surjective. 
We establish that this semi-conjugacy is finite-to-one, which
is the most elaborate part of the proof of Theorem \ref{coding}.
This is used to show that every ergodic $\Phi^t$-invariant
Borel probability measure on ${\cal Q}$ is the push-forward under
$\Xi$ of a finite invariant measure on $(X,\Theta^t)$ and hence
on $(\Omega,\sigma)$, which 
completes the proof of Theorem \ref{coding}. 

Theorem \ref{coding} is not sufficient for the proof of Theorem
\ref{entropymax}.
Namely, the semi-conjugacy $\Xi$ is only finite-to-one but not
bounded-to-one, and it is defined on a suspension over a Borel subset of
$\Omega$ which is not closed, reflecting the fact that the component 
${\cal Q}$ is not compact, and the Teichm\"uller flow
$\Phi^t$ is not hyperbolic. Instead, in Section \ref{sec:measureofmax}
we start with a point
$q\in {\cal Q}$ which is contained in both the $\alpha$- and $\omega$ limit
set of its own orbit under $\Phi^t$. By the Poincar\"e recurrence theorem,
for any invariant Borel probability measure on ${\cal Q}$, the set
of such points has full measure. We then use the point $q$ and the
subshift $(\Omega,\sigma)$ to construct a new Markov shift, now over a
countably infinite alphabet. We also construct a continuous roof function
which is bounded from below by a universal positive constant, but is
unbounded. We then show that
the corresponding suspension flow admits a bounded-to-one
semi-conjugacy onto the restriction of the Teichm\"uller flow $\Phi^t$
to the invariant set of all points whose orbit under $\Phi^t$ 
contain $q$ in its $\alpha$- and $\omega$ limit set.

It follows from the work of Buzzi and Sarig \cite{BS03} that
the suspension flow over the countable alphabet 
admits at most one measure of maximal entropy provided that
some technical conditions are fullfilled. We then verify that these
technical assumptions are indeed satisfied for the
flow constructed earlier, which leads to the proof of
Theorem \ref{entropymax}.

\paragraph{\bf Acknowledgement} This work 
was carried out during two visits of 
the MSRI in Berkeley (in fall 2007 and in fall 2011)  
and at the Hausdorff Institut in Bonn in spring 2010. 
I am very grateful to these two institutions for
their hospitality. I also thank the anonymous referees for 
very careful reading and many useful suggestions  to improve the
redaction of this article.

\section{Strata and train tracks}\label{sec:strata}

In this section we
summarize some results  from
\cite{H22}
which will be used throughout the paper.

\subsection{Geodesic laminations}
Let $S$ be an
oriented surface of
genus $g\geq 0$ with $m\geq 0$ marked points and where $3g-3+m\geq 2$.
A \emph{geodesic lamination} for a complete
hyperbolic structure of finite volume on $S$ (which means in the sequel that 
the metric is defined 
on the complement of the marked points in $S$) is
a \emph{compact} subset of $S$ (with the marked points removed) 
which is foliated into simple
geodesics.
A geodesic lamination $\lambda$ is called \emph{minimal}
if each of its half-leaves is dense in $\lambda$. 

A geodesic lamination $\lambda$ on $S$ is said to
\emph{fill up $S$} if its complementary regions 
are all topological discs or once
punctured monogons. 

\begin{definition}[Definition 2.1 of \cite{H22}]\label{large}
A geodesic lamination $\lambda$
is called \emph{large} if $\lambda$ fills up
$S$ and if
moreover $\lambda$ can be 
approximated in the \emph{Hausdorff topology}
by simple closed geodesics.
\end{definition}

Since every minimal geodesic lamination can be 
approximated in the Hausdorff topology by simple 
closed geodesics (Theorem 4.2.14 of \cite{CEG87}), a minimal geodesic
lamination which fills up $S$ is large. However, there are large
geodesic laminations with finitely many leaves.

The \emph{topological type} of a large geodesic
lamination $\nu$ is a tuple 
\[(m_1,\dots,m_\ell;-m)\text{ where }1\leq m_1\leq \dots \leq m_\ell,\,
\sum_{i}m_i=4g-4+m\]
such that the complementary regions of $\nu$ which are 
topological discs are $m_i+2$-gons and the complementary
regions which are once punctured discs are once punctured monogons. 
Let 
\[{\cal L\cal L}(m_1,\dots,m_\ell;-m)\]
be the space of all large geodesic laminations of type 
$(m_1,\dots,m_\ell;-m)$ equipped with the 
restriction of the Hausdorff
topology for compact subsets of $S$.

A \emph{measured geodesic lamination} is a geodesic lamination
$\lambda$ together with a translation invariant transverse
measure. Such a measure assigns a positive weight to each compact
arc in $S$ with endpoints in the complementary regions of
$\lambda$ which intersects $\lambda$ nontrivially and
transversely. The geodesic lamination $\lambda$ is called the
\emph{support} of the measured geodesic lamination; it consists of
a disjoint union of minimal components. The space ${\cal M\cal L}$
of all measured geodesic laminations on $S$ equipped with the
weak$^*$-topology is homeomorphic to $S^{6g-7+2m}\times
(0,\infty)$. Its projectivization is the space ${\cal P\cal M\cal
L}$ of all \emph{projective measured geodesic laminations}. 

The measured geodesic lamination $\mu\in {\cal
M\cal L}$ is said to \emph{fill up $S$} if its support fills up $S$.
This support is then necessarily connected and hence minimal,
and for some tuple $(m_1,\dots,m_\ell;-m)$,
it defines a point in the
set ${\cal L\cal L}(m_1,\dots,m_\ell;-m)$. 
The projectivization of a measured geodesic lamination
which fills up $S$ is also said to fill up $S$.
We call $\mu\in {\cal M\cal L}$ 
\emph{strongly uniquely ergodic}
if the support of $\mu$ fills up $S$ and admits a unique
transverse measure up to scale.

There
is a continuous symmetric pairing 
\[\iota:{\cal M\cal L}\times {\cal M\cal L}\to [0,\infty),\] 
the so-called \emph{intersection form},
which extends the geometric intersection number between simple
closed curves. We refer to Section 8.2.10 of \cite{Mar16} for
more information.

\subsection{Train tracks}\label{sec:train}

A \emph{train track} on $S$ is defined to be an embedded
1-complex $\tau\subset S$ whose edges
(called \emph{branches}) are smooth arcs with
well-defined tangent vectors at the endpoints. At any vertex
(called a \emph{switch}), the incident edges are mutually tangent.
Through each switch there is a path of class $C^1$
which is embedded
in $\tau$ and contains the switch in its interior. 
A simple closed curve component of $\tau$ is required to contain
a unique bivalent switch, and all other switches must 
be at least trivalent.
The complementary regions of the
train track have negative Euler characteristic, which means
that they are different from discs with $0,1$ or
$2$ cusps at the boundary and different from
annuli and once-punctured discs
with no cusps at the boundary.
We always identify train
tracks which are isotopic.
Throughout we use the book \cite{PH92} as the main reference for 
train tracks. 

A train track is called \emph{generic} if all switches are
at most trivalent. For each switch $v$ of 
a generic train track $\tau$ which is not contained in 
a simple closed curve component, there is a unique
half-branch $b$ of $\tau$ which is incident on $v$ and which is
\emph{large} at $v$. This means that every germ of an immersed
arc of class $C^1$ on $\tau$ which passes through $v$ also
passes through the interior of $b$. 
A half-branch which is not large is called \emph{small}.
A branch $b$ of $\tau$ is
called \emph{large} (or \emph{small}) if each of its
two half-branches is large (or small). A branch which 
is neither large nor small is called \emph{mixed}.

\begin{remark}
As in \cite{H09a}, all train tracks
are assumed to be generic. Unfortunately this leads to 
a small inconsistency of our terminology with the
terminology found in the literature.
\end{remark}


A generic 
train track $\tau$ is \emph{orientable} 
if there is a consistent orientation of the 
branches of $\tau$ such that 
at any switch $s$ of $\tau$, the orientation of the large
half-branch incident on $s$ extends to the orientation
of the two small half-branches incident on $s$.
If $C$ is a complementary polygon of an oriented
train track, then the number of sides of $C$ is even.
In particular, a train track which contains a once
punctured monogon component,  that is,
a once punctured disc with
one cusp at the boundary, is not orientable
(see p.31 of \cite{PH92} for 
a more detailed discussion).

A train track or a geodesic lamination $\eta$ is
\emph{carried} by a train track $\tau$ if
there is a map $F:S\to S$ of class $C^1$ which is homotopic to the
identity and maps $\eta$ into $\tau$ in such a way 
that the restriction of the differential of $F$
to the tangent space of $\eta$ vanishes nowhere;
note that this makes sense since a train track has a tangent
line everywhere. We call the restriction of $F$ to
$\eta$ a \emph{carrying map} for $\eta$.
Write $\eta\prec
\tau$ if the train track $\eta$ is carried by the train track
$\tau$. Then every geodesic lamination $\nu$ which is carried
by $\eta$ is also carried by $\tau$.

A train track \emph{fills up} $S$ if its complementary
components are topological discs or once punctured 
monogons.  Note that such a train track
$\tau$ is connected.
Let $\ell\geq 1$ be the number of those complementary 
components of $\tau$ which are topological discs.
Each of these discs is an $m_i+2$-gon for some $m_i\geq 1$
$(i=1,\dots,\ell)$. The
\emph{topological type} of $\tau$ is defined to be
the ordered tuple $(m_1,\dots,m_\ell;-m)$ where
$1\leq m_1\leq \dots \leq m_\ell$; then $\sum_im_i=4g-4+m$.
If $\tau$ is orientable then $m=0$ and $m_i$ is even 
for all $i$. A train track of topological type $(m_1,\dots,m_\ell;-m)$
is called \emph{fully recurrent} if it carries a minimal large geodesic lamination
of type $(m_1,\dots,m_\ell;-m)$ (Definition 2.6 of \cite{H22}). 

A \emph{transverse measure} on a generic train track $\tau$ is a
nonnegative weight function $\mu$ on the branches of $\tau$
satisfying the \emph{switch condition}:
for every trivalent switch $s$ of $\tau$,  the sum of the weights
of the two small half-branches incident on $s$ 
equals the weight of the large half-branch.
The space 
\[{\cal V}(\tau)\] of all measured geodesic laminations whose 
supports are carried by
$\tau$ can naturally be identified with the space of all 
transverse measures
on $\tau$. Thus ${\cal V}(\tau)$ 
has the structure of a cone in a finite dimensional
real vector space.
The train track is called
\emph{recurrent} if it admits a transverse measure which is
positive on every branch.  A fully recurrent train track is recurrent
\cite{PH92,H22}.



There are two simple ways to modify a fully recurrent
train track $\tau$
to another fully recurrent train track.
Namely, if $b$ is a mixed branch of $\tau$ 
then we can \emph{shift} $\tau$
along $b$ to a new train track $\tau^\prime$. 
This new train track carries $\tau$ and hence it 
is fully recurrent since it carries 
every geodesic lamination
which is carried by $\tau$,
see p.118 of \cite{PH92} and also \cite{H09a}.

Similarly, 
if $e$ is a large branch of $\tau$ then we can perform a
right or left \emph{split} of $\tau$ at $e$
as shown in Figure A below.
A (right or left) split $\tau^\prime$ of a 
train track $\tau$ is carried
by $\tau$. 
If $\tau$ is of topological type 
$(m_1,\dots,m_\ell;-m)$, 
if $\nu\in {\cal L\cal L}(m_1,\dots,m_\ell;-m)$
is carried by $\tau$ and if $e$ is a large branch
of $\tau$, then there is a unique choice of a right or
left split of $\tau$ at $e$ such that the split track $\eta$ 
carries $\nu$. In particular, $\eta$ is fully recurrent. 
We refer to p.48-49 of \cite{H22} for a more detailed discussion.
Note that unlike in the standard reference \cite{PH92}, we do \emph{not}
allow to modify a train track with a \emph{collision}, which is defined to be
a split followed by the removal of the diagonal branch of the split and
which strictly reduces the number of branches of the train track.
\begin{figure}[ht]
\begin{center}
\psfrag{e}{$e$}
\psfrag{b}{$c_3$}
\psfrag{c}{$c_5$}
\psfrag{Figure A}{Figure A} 
\includegraphics[width=0.8\textwidth]{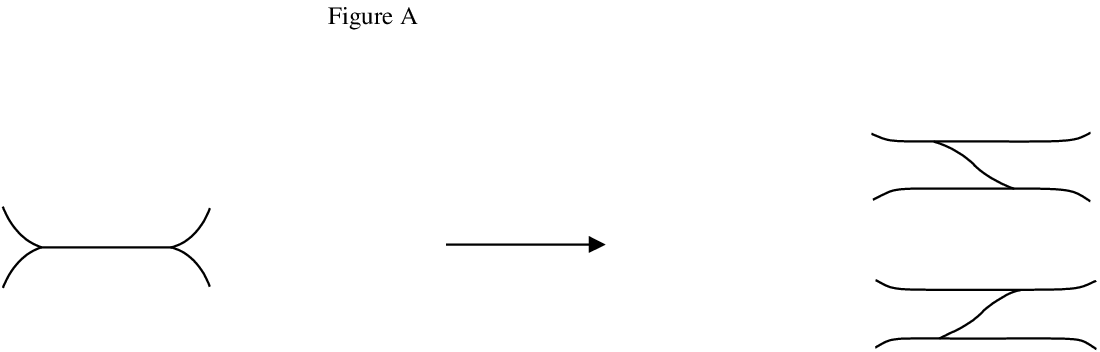}
\end{center}
\end{figure}

To each train track $\tau$ 
which fills up $S$ one can
associate a \emph{dual bigon track} $\tau^*$ 
(Section 3.4 of \cite{PH92}).
There is a bijection between
the complementary components of $\tau$ and those
complementary components of $\tau^*$ which are
not \emph{bigons}, i.e. discs with two cusps at the
boundary. This bijection maps
a component $C$ of $\tau$ which is an $n$-gon for some
$n\geq 3$ to an $n$-gon component of 
$\tau^*$ contained in $C$, and it maps a once punctured
monogon $C$ to a once punctured monogon contained in $C$.
If $\tau$ is orientable then the orientation of $S$ and
an orientation of $\tau$ induce an orientation on 
$\tau^*$, that is, $\tau^*$ is orientable.

There is a notion of carrying for bigon tracks which 
is analogous to the notion of carrying for train tracks. 

\begin{definition}[Definition 2.8 of \cite{H22}]\label{largett}
A train track $\tau$ of topological type \\
$(m_1,\dots,m_\ell;-m)$ 
is called \emph{large} if 
both $\tau,\tau^*$ carry a minimal  
large geodesic lamination of the same topological type as $\tau$. 
\end{definition}

For a large train track $\tau$ let 
${\cal V}^*(\tau)\subset {\cal M\cal L}$ 
be the set of all measured geodesic
laminations whose support is carried by $\tau^*$. 
Then 
${\cal V}^*(\tau)$ is naturally homeomorphic to 
a convex cone in a real vector space. The dimension of this cone
coincides with the dimension of ${\cal V}(\tau)$.

Denote by 
\[{\cal L\cal T}(m_1,\dots,m_\ell;-m)\] the set of all
isotopy classes of large train tracks on $S$ of type
$(m_1,\dots,m_\ell;-m)$.

\subsection{Strata}\label{subsec:strata}

For a closed oriented surface 
$S$ of genus $g\geq 0$ with $m\geq 0$ punctures
let $\tilde{\cal Q}(S)$ be the
bundle of marked area one holomorphic
quadratic differentials with a simple pole at each puncture
over the Teichm\"uller space ${\cal T}(S)$ 
of marked complex structures
on $S$.
For a complete hyperbolic metric on $S$ of
finite area, an area one quadratic differential 
$q\in \tilde{\cal Q}(S)$ is 
determined by a pair $(\lambda^+,\lambda^-)$ of 
measured geodesic laminations 
which \emph{bind} $S$, that is, 
we have 
\[\iota(\lambda^+,\mu)+
\iota(\lambda^-,\mu)>0\] for every measured geodesic
lamination $\mu$, moreover  
$\iota(\lambda^+,\lambda^-)=1$ as the area of $q$ equals one. 
The \emph{vertical} measured geodesic 
lamination $\lambda^+$ for $q$
corresponds to the equivalence class of the vertical measured
foliation of $q$. 
The \emph{horizontal} measured geodesic lamination
$\lambda^-$ for $q$ corresponds to the equivalence
class of the horizontal measured foliation of $q$.

Recall from the introduction the definition of the
\emph{stratum} $\tilde{\cal Q}(m_1,\dots,m_\ell;-m)$ of 
$\tilde{\cal Q}(S)$
where $(m_1,\dots,m_\ell)$ of positive integers 
$1\leq m_1\leq \dots \leq m_\ell$ with $\sum_im_i=4g-4+m$.
Also consider as before the bundle 
$\tilde{\cal H}(S)$ of marked area one
holomorphic one-forms over Teichm\"uller space
${\cal T}(S)$ of $S$ with its strata 
$\tilde{\cal H}(k_1,\dots,k_\ell)$ of marked area one holomorphic one-forms
on $S$ with $\ell$ zeros of order $k_i$ $(i=1,\dots,\ell)$.

%
%
%
%
For a large train track $\tau\in {\cal L\cal T}(m_1,\dots,m_\ell;-m)$ 
let 
\begin{equation}\label{qtau}
{\cal Q}(\tau)\subset \tilde{\cal Q}(S)\end{equation} be 
the set of all marked area one quadratic differentials 
whose vertical measured geodesic lamination  
is carried by $\tau$ and whose horizontal 
measured geodesic lamination is carried by the dual
bigon track $\tau^*$ of $\tau$.
By definition of a large train track, 
we have ${\cal Q}(\tau)\not=\emptyset$.

The next proposition relates 
${\cal Q}(\tau)$ to components of strata. 

\begin{proposition}[Proposition 3.2 and Proposition 3.3 of \cite{H22}]\label{structure}
\begin{enumerate}
\item 
For any large non-orientable train track
$\tau\in {\cal L\cal T}(m_1,\dots,m_\ell;-m)$ 
there is a component $\tilde{\cal Q}$ 
of the stratum
$\tilde{\cal Q}(m_1,\dots,m_\ell;-m)$ such that
${\cal Q}(\tau)$ is the closure 
in $\tilde{\cal Q}(S)$  of 
an open path connected subset of $\tilde{\cal Q}$.
\item For every large orientable train track 
$\tau\in {\cal L\cal T}(m_1,\dots,m_\ell;0)$ 
there is a component $\tilde{\cal Q}$ of the stratum
$\tilde{\cal H}(m_1/2,\dots,m_\ell/2)$
such that ${\cal Q}(\tau)$ is the closure in $\tilde {\cal H}(S)$ of 
an open path connected subset of $\tilde{\cal Q}$.
\item For every component $\tilde {\cal Q}$ of a stratum
of $\tilde {\cal Q}(S)$ (or of a stratum of $\tilde {\cal H}(S)$) 
and for every $q\in \tilde {\cal Q}$ there is a large
train track $\tau$ such that $q\in {\cal Q}(\tau)$ and 
that $Q(\tau)$ is the closure 
of the open dense path connected subset  
$\tilde {\cal Q}\cap {\cal Q}(\tau)$.
\end{enumerate}
\end{proposition}


\section{A symbolic system}\label{sec:asymbolic} 

In this section we
construct a subshift of finite type which is 
used in the following sections to construct a symbolic coding
for the Teichm\"uller flow on a component 
of a stratum 
in the moduli space of quadratic or abelian differentials.
We continue to use the assumptions and notations
from Section \ref{sec:strata}. The section is divided into two subsections.

\subsection{A shift space constructed from train tracks}\label{sec:ashift}
Let ${\cal Q}$ be a connected component of a 
stratum ${\cal Q}(m_1,\dots,m_\ell;-m)$ of ${\cal Q}(S)$
(or of a stratum
${\cal H}(m_1/2,\dots,m_\ell/2)$ of ${\cal H}(S)$).  Let $\tilde {\cal Q}$
be the preimage of ${\cal Q}$ in $\tilde{\cal Q}(S)$
(or in $\tilde {\cal H}(S)$).
Let 
\begin{equation}\label{lt}
{\cal L\cal T}({\cal Q})\subset {\cal L\cal T}(m_1,\dots,m_\ell;-m)\end{equation}
be the set of all marked large train tracks $\tau$ of the same 
topological type as ${\cal Q}$ 
such that the set 
${\cal Q}(\tau)$ defined in (\ref{qtau}) is the
closure of the open dense subset
${\cal Q}(\tau)\cap \tilde{\cal Q}$.
Proposition \ref{structure}
shows that this is well defined and that furthermore 
\[\tilde {\cal Q}=\cup_{\tau\in {\cal L\cal T}({\cal Q})} ({\cal Q}(\tau)\cap \tilde {\cal Q}).\]
The set ${\cal L\cal T}({\cal Q})$ is
invariant under the action of the mapping class group.


Fix a complete finite volume metric on $S$. 
For ease of notation, define 
\begin{equation}\label{ll}
{\cal L\cal L}({\cal Q})\subset {\cal L\cal L}(m_1,\dots,m_\ell;-m)\end{equation}
to be the closure (in the restriction of the 
Hausdorff topology defined by the hyperbolic metric) of the 
set of all minimal large geodesic laminations which 
can be represented as the support of the  
vertical measured geodesic lamination
of some  quadratic differential $q\in \tilde {\cal Q}$. 
Then ${\cal L\cal L}({\cal Q})$ is invariant under the 
action of ${\rm Mod}(S)$, and for every 
$\tau\in {\cal L\cal T}({\cal Q})$ it contains the set of all 
large geodesic laminations of topological type 
$(m_1,\dots,m_\ell;-m)$ carried by $\tau$.
We refer to Section 3 of \cite{H22} for a more detailed discussion.

Our goal is to use the train tracks from the collection 
${\cal L\cal T}({\cal Q})$  for a symbolic
coding of the Teichm\"uller flow on ${\cal Q }$.
However, the mapping class
group ${\rm Mod}(S)$ does not act freely
on ${\cal L\cal T}({\cal Q})$. 
To overcome this difficulty
we extend 
the definition of a large train track 
as follows.

\begin{definition}\label{numberedlarge}
A \emph{numbered marked large train track}
is a marked large train track $\tau$ together with
a numbering of the branches of $\tau$.
\end{definition}

The set 
\begin{equation}\label{numb}
{\cal N\cal T}({\cal Q})\end{equation} of all isotopy
classes of numbered marked large train tracks on $S$
whose underlying unnumbered large train track is 
contained in ${\cal L\cal T}({\cal Q})$
is invariant
under the natural action of the mapping class group.

A mapping class which preserves a marked large train track
$\tau$ as well as each of
its branches is the identity. Namely, such a mapping class 
can be represented by a homeomorphism of $S$ whose restriction to 
$\tau$ is the identity. Since all complementary components of 
$\tau$ are discs or once punctured discs, such a homeomorphism is 
homotopic to the identity. Thus the action of the 
mapping class group on ${\cal N\cal T}({\cal Q})$ is free.

Define 
a \emph{(numbered) combinatorial type}
to be an orbit of a (numbered) marked large train track in ${\cal L\cal T}({\cal Q})$
(or in ${\cal N\cal T}({\cal Q})$)
under
the action of the mapping class group.
Thus the set  
of numbered combinatorial
types is the quotient of ${\cal N\cal T}({\cal Q})$ 
by the action of the
mapping class group.
Let \[{\cal E}_0({\cal Q})\]
be the set of all numbered combinatorial types which 
are ${\rm Mod}(S)$-orbits of elements of 
${\cal N\cal T}({\cal Q})$.
%


Note that if the large train track $\tau^\prime$
can be obtained from a 
large train track $\tau$ by
a single split, then a numbering of the
branches of $\tau$ naturally induces a numbering of the
branches of $\tau^\prime$ and therefore
such a numbering defines a
\emph{numbered split}. 

\begin{definition}\label{fullsplit}
A \emph{full split} of a (numbered) 
large train track $\tau$ is a (numbered)
large train track $\tau^\prime$ which 
can be obtained from $\tau$ by splitting $\tau$ at each
large branch precisely once.
\end{definition}

A \emph{full (numbered) splitting sequence} 
is a sequence $(\tau_i)$ of (numbered)
large train tracks such that for each $i$, the (numbered)
large train track $\tau_{i+1}$ can be obtained from
$\tau_i$ by a full (numbered) split.

\begin{definition}\label{def:splittable}
A numbered combinatorial type $x\in {\cal E}_0({\cal Q})$ is
\emph{splittable} to a numbered combinatorial type $x^\prime$
if there is a numbered large train track $\tau$ 
contained in $x$ 
which can be
connected to a numbered large
train track $\tau^\prime$ contained in $x^\prime$ by a
\emph{full} numbered splitting sequence. 
\end{definition}

In general it is unclear whether a given numbered 
combinatorial type is splittable to another type. This issue is addressed in 
Lemma \ref{numberedtype} below which is a main technical
ingredient towards the construction of a subshift of finite type with the
properties stated in Theorem \ref{coding}. 

Given a numbered marked train track $\tau$ and a subset ${\cal W}$ of 
${\cal E}_0({\cal Q})$, we write $[\tau]\in {\cal W}$ if 
the ${\rm Mod}(S)$-orbit of $\tau$ is contained in ${\cal W}$. 
The first statement in the following lemma is crucial for topological 
transitivity of the subshift of finite type which defines our coding
of the Teichm\"uller flow.

\begin{lemma}\label{numberedtype} 
For every connected component ${\cal Q}$ of a stratum of ${\cal Q}(S)$
(or of a stratum of ${\cal H}(S)$)
there is a set ${\cal E}({\cal Q})
\subset {\cal E}_0({\cal Q})$ 
of numbered combinatorial types with the following properties.
\begin{enumerate}
\item For all $x,x^\prime\in 
{\cal E}({\cal Q})$, $x$ is splittable to $x^\prime$.
\item If $\tau$ is contained in ${\cal E}({\cal Q})$ and
if $(\tau_i)$ is any full numbered splitting
sequence issuing from $\tau_0=\tau$ then $\tau_i$ is
contained in ${\cal E}({\cal Q})$
for all $i\geq 0$.
\end{enumerate}
\end{lemma}
\begin{proof}
For $[\beta]\in {\cal E}_0({\cal Q})$ 
let ${\cal A}([\beta])\subset {\cal E}_0({\cal Q})$ be the set of 
all combinatorial types of numbered large 
train tracks $\xi$ with the following additional property.
There is 
a representative $\beta\in {\cal N\cal T}({\cal Q})$ of $[\beta]$ 
which can be connected to $\xi$ by  
a (possibly trivial) full numbered splitting sequence.
Since the concatenation of two 
full splitting sequences is a full splitting sequence, if 
$[\xi]\in {\cal A}([\beta])$ then ${\cal A}([\xi])\subset {\cal A}([\beta])$.

Since ${\cal E}_0({\cal Q})$ is a finite set, there exists 
some $[\sigma]\in {\cal E}_0({\cal Q})$ such that the cardinality of 
${\cal A}([\sigma])$ is minimal among the cardinalities of the 
sets ${\cal A}([\beta])$ where $[\beta]$ ranges through ${\cal E}_0({\cal Q})$.
Let $[\xi]\in {\cal A}([\sigma])$. Since ${\cal A}([\xi])\subset {\cal A}([\sigma])$ and 
the cardinality of ${\cal A}([\sigma|)$ is minimal, we know that 
${\cal A}([\xi])= {\cal A}([\sigma])$ and, in particular, 
$[\sigma]\in {\cal A}([\xi])$. As $[\xi]\in {\cal A}([\sigma])$ was arbitrary, we conclude 
that 
for all $\xi,\xi^\prime\in {\cal N\cal T}({\cal Q})$ which are
contained  in ${\cal A}([\sigma])$, there
is a full numbered splitting sequence connecting
$\xi$ to a train track $g\xi^\prime$ for some $g\in {\rm Mod}(S)$.

Define 
\begin{equation}
{\cal E}({\cal Q})={\cal A}([\sigma])\subset {\cal E}_0({\cal Q}).\end{equation}
Let $\theta\in {\cal N\cal T}({\cal Q})$ with  
$[\theta]\in {\cal E}({\cal Q})$.
By the above discussion, 
the train track $\theta$ can be connected to a
numbered train track $\sigma^\prime$ in the ${\rm Mod}(S)$-orbit of $\sigma$ 
by a full numbered splitting sequence.
Thus the first property in the lemma holds
true for ${\cal E}({\cal Q})$, and the second
is true by the definition of ${\cal E}({\cal Q})$. 
This completes the proof of the lemma.
\end{proof}

Let $k>0$ be the cardinality
of the set ${\cal E}({\cal Q})\subset {\cal E}_0({\cal Q})$ as in 
Lemma \ref{numberedtype}. 
Number the $k$ elements of ${\cal E}({\cal Q})$ in an arbitrary
way. Identify each element of ${\cal E}({\cal Q})$ with its number.
Define $a_{ij}=1$ if the numbered combinatorial type $i$
can be split with a single full numbered
split to the numbered combinatorial type
$j$ and define $a_{ij}=0$ otherwise.
The matrix $A=(a_{ij})$ defines
a \emph{subshift of finite type}. Its phase space
is the set of biinfinite sequences
\[\Omega=\{(x_i)\in \prod_{i=-\infty}^\infty
\{1,\dots,k\} \mid a_{x_ix_{i+1}}=1 \text{ for all } i\}.\] 
Every biinfinite full numbered splitting
sequence $(\tau_i)\subset {\cal N\cal T}({\cal Q})$
contained in ${\cal E}({\cal Q})$ defines a
point in $\Omega$. Vice versa, since 
the action
of ${\rm Mod}(S)$ on the set of numbered
large train tracks is free, a point $(x_i)\in \Omega$
determines a ${\rm Mod}(S)$-orbit
of biinfinite full numbered splitting sequences.
We say that such a full numbered
splitting sequence \emph{realizes} $(x_i)$.

The shift map $\sigma:\Omega\to \Omega, \sigma(x_i)=(x_{i+1})$
acts on $\Omega$.
For $n>0$ write
$A^n=(a_{ij}^{(n)})$; the shift $\sigma$
is \emph{topologically transitive}
if for all $i,j$ there is some $n >0$ such that
$a_{ij}^{(n)}>0$. 
Define a finite sequence $(x_i)_{0\leq i\leq n}$ of
points $x_i\in \{1,\dots,k\}$ to be
\emph{admissible} if $a_{x_ix_{i+1}}=1$ for all $i\leq n-1$.
Then $a_{ij}^{(n)}$ equals the number of all
admissible sequences of length $n$ connecting $i$ to $j$
\cite{Mn87}.
The following
observation is immediate from the definitions.

\begin{lemma}\label{shift} The shift $(\Omega,\sigma)$ is topologically
	transitive.
\end{lemma}
\begin{proof}
	Let $i,j\in \{1,\dots,k\}$ be arbitrary.
	By Lemma \ref{numberedtype}, there is a nontrivial
	finite full numbered splitting sequence
	$\{\tau_i\}_{0\leq i\leq n}\subset {\cal N\cal T}({\cal Q})$
	connecting a
	train track $\tau_0$ of numbered combinatorial type
	$i$ to a train track $\tau_n$ of
	numbered combinatorial type $j$.
	This splitting sequence then
	defines an admissible 
	sequence $(x_i)_{0\leq i\leq n}$ connecting
	$i$ to $j$. \end{proof}

\begin{remark}
	Without loss of generality, we can in fact
	assume that the shift $(\Omega,\sigma)$ is \emph{topologically 
		mixing}, that is, there exists a number $n>0$ so that
	the matrix $A^n$ is positive. Namely, by the discussion on p.55 of 
	\cite{HK95}, otherwise there are numbers $\ell,n>0$ such that
	$\ell n=k$ and that the following holds true.
	The elements of 
	${\cal E}({\cal Q})$ are divided into $n$ disjoint
	sets $C_1,\dots,C_n$ 
	of $\ell$ elements each so that for 
	$x_i\in C_j$ we have $a_{x_ix_{i+1}}=1$ only
	if $x_{i+1}\in C_{j+1}$ (indices are taken modulo $n$).
	Moreover, the restriction of $\sigma^n$ to $C_1$ 
	is topologically mixing. However, in this case we can 
	repeat the argument in the proof of Lemma \ref{numberedtype}
	with a single full numbered split replaced by a full numbered
	splitting sequence of length $n$. This amounts to replacing
	$(\Omega,\sigma)$ by the topologically mixing subshift
	$(C_1,\sigma^n)$ which 
	has all properties stated above.
\end{remark}

\subsection{Relation to the Teichm\"uller flow}\label{sec:relteich}
In this subsection we connect the subshift of finite type $(\Omega,\sigma)$ constructed in 
Subsection \ref{sec:ashift} to the Teichm\"uller flow on ${\cal Q}$.
To this end 
we shall make use of the following simple consequence of a celebrated result of Masur
(Theorem 1.1 of \cite{M92})
about the structure of differentials
$q\in {\cal Q}$ which are recurrent under the Teichm\"uller flow. 

\begin{lemma}\label{minimalfull}
Let $q\in {\cal Q}$ be a point whose forward orbit $\Phi^tq$ $(t\geq 0)$ 
under the Teichm\"uller flow returns to a compact set
$K\subset {\cal Q}$ for arbitrarily large times.
Then the vertical measured
geodesic lamination of a lift of $q$ to $\tilde {\cal Q}$ is 
uniquely ergodic, with support in 
${\cal L\cal L}({\cal Q})$.
\end{lemma}
\begin{proof}  Unique ergodicity of the vertical measured geodesic lamination 
of a forward recurrent 
differential $q \in {\cal Q}$
is Theorem 1.1 of \cite{M92}. 

To see that  the support of the measured geodesic lamination  of a lift 
$\tilde q\in \tilde {\cal Q}$ of $q$ is contained in 
${\cal L\cal L}({\cal Q})$ assume otherwise. Then
$q$ admits at least one vertical saddle connection.
This saddle connection has finite length. 

By continuity, for every compact set $K\subset {\cal Q}$ there exists 
a number $\kappa(K)>0$ which bounds from below the minimal length of 
any saddle connection for a differential $u\in K$.
Now the length of a vertical
saddle connection is exponentially decreasing under the 
Teichm\"uller flow and therefore the orbit of $q$ does not
return to $K$ for arbitrarily large times. This completes the proof.
\end{proof}

The union 
$\cup_{[\tau]\in  {\cal E}({\cal Q})} {\cal Q}(\tau)\cap \tilde {\cal Q}$
 is clearly invariant under the Teichm\"uller flow and under the action of
 the mapping class group, and it contains a non-empty open invariant subset of $\tilde {\cal Q}$.
 However, this does not imply that it projects to a subset of ${\cal Q}$ 
 which contains the support of every $\Phi^t$-invariant Borel probability measure. 
 The next lemma provides the relevant additional information we need,

\begin{lemma}\label{numberedtype2}
For every $q\in {\cal Q}$ 
without 
vertical saddle connections 
and every lift
$\tilde q$ of $q$ to $\tilde {\cal Q}$ 
there is some $\tau\in {\cal N\cal T}({\cal Q})$ such that
$[\tau]\in {\cal E}({\cal Q})$ and 
$\tilde q\in {\cal Q}(\tau)$.
\end{lemma}
\begin{proof}
Let $q\in {\cal Q}$ and 
assume that $q$ does not have 
vertical saddle connections. 
Let $\tilde q\in \tilde {\cal Q}$
be a lift of $q$. 
By Proposition \ref{structure} and using the notation (\ref{lt}), 
there is a train track $\eta\in {\cal L\cal T}({\cal Q})$
so that 
$\tilde q\in {\cal Q}(\eta)$.
Recall that ${\cal Q}(\eta)\cap \tilde {\cal Q}$ contains an open subset of  
 $\tilde {\cal Q}$ which projects to an open subset of 
 ${\cal Q}$. Since the set of all points in 
${\cal Q}$ whose $\Phi^t$-orbits are dense in ${\cal Q}$ has
full Lebesgue measure and hence is dense, there is 
a point $z\in {\cal Q}$ whose  
$\Phi^t$-orbit is dense 
in ${\cal Q}$ in forward and backward direction and which
admits a lift 
$\tilde z\in {\cal Q}(\eta)$.
 
By the definition of ${\cal Q}(\eta)$, 
the support of the horizontal measured geodesic
lamination $\tilde z^h$ of $\tilde z$ is  carried
by the dual $\eta^*$ of $\eta$. Furthermore, by Lemma \ref{minimalfull},
applied to the time reversal $t\to \Phi^{-t}$ of the Teichm\"uller flow,  
the support of $\tilde z^h$ fills $S$.
Thus since the support of the vertical 
measured geodesic lamination $\tilde q^v$ of $\tilde q$ is carried by 
$\eta$, the measured geodesic laminations $\tilde q^v,\tilde z^h$ bind $S$.
Hence there is a 
quadratic differential $\tilde u\in {\cal Q}(\eta)$ 
whose horizontal measured
geodesic lamination equals $\tilde z^h$ and whose
vertical measured geodesic lamination 
equals $c\tilde q^v$ for a number $c>0$. 
Proposition \ref{structure} implies that
$\tilde u\in \tilde {\cal Q}$.
 
Since the orbit of $z$ under the time reversal of the 
Teichm\"uller flow
is dense in ${\cal Q}$, and since $\tilde u,\tilde z$ have the same horizontal 
measured laminations, Theorem 2 of \cite{M80} shows that 
we have $d(\Phi^{-t}\tilde u,\Phi^{-t}\tilde z)\to 0$
$(t\to \infty)$ for any distance function $d$ on $\tilde {\cal Q}$ which is 
induced by a
complete ${\rm Mod}(S)$-invariant
Riemannian metric on 
$\tilde{\cal Q}$.

On the other hand, let $\sigma \in {\cal N\cal T}({\cal Q})$ 
be contained in ${\cal E}({\cal Q})$.
Then ${\cal Q}(\sigma)$ contains an open subset of 
$\tilde {\cal Q}$.  Using once more the fact that
the backward orbit of 
$z$ is dense in ${\cal Q}$, there is some
$g\in {\rm Mod}(S)$ and some $T>0$ 
so that $g \Phi^{-T}\tilde u\in 
{\cal Q}(\sigma)$ and hence $g\tilde u\in {\cal Q}(\sigma)$ 
by invariance.  
In particular, the vertical measured geodesic
lamination $g(c\tilde q^v)$ of $g\tilde u$ is carried by
$\sigma$. Equivalently, $\tilde q^v$ is carried by $g^{-1}\sigma$.

Now $\tilde q$ does not have any vertical saddle connection by assumption 
and hence the complementary components of the vertical measured 
geodesic lamination $\tilde q^v$ of $\tilde q$ are in bijection with the
zeros of $\tilde q$. Then 
$\tilde q^v$ is minimal,
with support in the set ${\cal L\cal L}({\cal Q})$ defined in (\ref{ll}).
Together with Lemma 5.1 of \cite{H09a}, we 
obtain that 
there is a unique full numbered splitting sequence 
$(\tau_i)\subset {\cal N\cal T}({\cal Q})$ beginning at
$\tau_0=g^{-1}\sigma$ which consists of train tracks
carrying $\tilde q^v$. By the second statement in Lemma \ref{numberedtype}, 
each of the train tracks
$\tau_i$ is contained in ${\cal E}({\cal Q})$. 

We claim that for sufficiently large $i$ we have $\tilde q\in {\cal Q}(\tau_i)$.
This then completes the proof of the lemma. To this end 
note that $\tilde q^v$ is carried by both $\eta,\tau_0$, and the 
topological type of its support coincides with the topological type of 
$\eta,\tau_0$. Thus by Corollary 2.4.3 of \cite{PH92}, there exists a train track 
$\nu$ which can be obtained from both $\eta,\tau_0$ by a splitting and 
collision sequence (where a collision is a split followed by the removal of the 
diagonal) 
and which carries $\tilde q^v$. Since the topological type of $\tilde q^v$ coincides with the
topological type of $\eta$ and $\tilde q^v$ is carried by $\nu$, the 
 topological type of $\nu$ coincides 
with the topological type of $\eta,\tau_0$ and hence there is no collision in 
the transformation of $\tau_0$ to $\nu$. 

Now by uniqueness of splitting sequences as established
in Lemma 5.1 of \cite{H09a} and the fact that the (numbered) 
splitting sequence $(\tau_i)$ is full, 
there exists some $j>0$ such that
$\nu$ is splittable to $\tau_j$ and hence $\eta$ is splittable to $\tau_j$. 
As $\tilde q\in {\cal Q}(\eta)$, the horizontal 
measured geodesic lamination $\tilde q^h$ of $\tilde q$ is carried by 
$\eta^*$ and hence by $\tau_j^*$. But this means that $\tilde q\in {\cal Q}(\tau_j)$.
As $[\tau_j]\in {\cal E}({\cal Q})$, this completes the proof of the lemma.
\end{proof}

\begin{remark}
We do not know whether there are components
${\cal Q}$ with ${\cal E}({\cal Q})=
{\cal E}_0({\cal Q})$. It seems likely that such components do not exist as
we expect that for $\eta\in {\cal E}({\cal Q})$, not all permutations 
of the numberings of the branches of $\eta$ are contained in ${\cal E}({\cal Q})$.
\end{remark}

\section{Symbolic dynamics for the Teichm\"uller flow}\label{sec:symbolicdyn}

In this section we relate the subshift of finite
type $(\Omega,\sigma)$ constructed in Section \ref{sec:asymbolic} to the
Teichm\"uller flow $\Phi^t$ on the component ${\cal Q}$ of a stratum. 
The section is divided into four subsections. 

In Section \ref{sec:rooffunction}, we 
define a $\sigma$-invariant Borel subset ${\cal U}_+$ of the biinfinite shift space $(\Omega,\sigma)$ called 
\emph{uniquely
ergodic sequences}, and we define a continuous 
bounded roof function $\rho$  
on the set ${\cal U}_+$ only depending on the future. The set ${\cal U}_+$ will be a superset of the 
Borel set ${\cal U}$ which appears in Theorem \ref{coding}, but  
a priori, the set may be empty. 

In Section \ref{sec:normal}, we establish that normal sequences in $(\Omega,\sigma)$ are contained in 
${\cal U}_+$. We then choose ${\cal U}$ to be the intersection of ${\cal U}_+$ with its
mirror image defined by time reversal of the shift. All normal sequences are contained in 
${\cal U}$, in particular, ${\cal U}$ is not empty, and the suspension of ${\cal U}$ with respect to 
the roof function $\rho$ is defined.

In Section \ref{sec:semiconj} we construct a map ${\cal U}\to {\cal Q}$ and establish
that is defines a finite-to one semi-conjugacy $\Xi$ of the suspension flow over ${\cal U}$ 
with roof function $\rho$ onto a
$\Phi^t$-invariant Borel subset of ${\cal Q}$. Finally in Section \ref{sec:suronmeas} we study the 
action of $\Xi$ on invariant measures for the suspension flow and complete the proof of 
Theorem \ref{coding}.
Throughout, 
we continue to use the 
assumptions and notations from Section \ref{sec:strata} and Section \ref{sec:asymbolic}.

\subsection{A roof function for the shift space}\label{sec:rooffunction}

Let as before ${\cal Q}$ be a component of a stratum.
Let ${\cal N\cal T}({\cal Q})$ be as in (\ref{numb}) and let $p>0$ be
the number of branches of a train track in ${\cal N\cal T}({\cal Q})$.
For $\tau\in {\cal N\cal T}({\cal Q})$ 
let as before ${\cal V}(\tau)$ be the space of 
all measured geodesic laminations carried by $\tau$,
equipped with the topology as the subcone of $\mathbb{R}^p$ 
of nonnegative 
functions on the branches of $\tau$ which satisfy the switch condition.
For each $\mu\in {\cal V}(\tau)$ we denote by 
\begin{equation}\label{totalmass}  
\mu(\tau)=\sum_{b\subset \tau}\mu(b)\end{equation} 
the total mass
of $\mu$, that is, the sum of the weights of $\mu$ over all branches 
of $\tau$. Define 
\[{\cal V}^{\cal P}(\tau)\subset {\cal V}(\tau)\]
to be the subspace of transverse (probability) measures of total mass $1$.

Denote by $P{\cal V}(\tau)$ 
the space of all \emph{projective} measured 
geodesic laminations which are carried by $\tau$. 
Note that $P{\cal V}(\tau)$ is a
\emph{compact} subset of the compact space ${\cal P\cal M\cal L}$
of all projective measured geodesic laminations on $S$. 

Let $(\tau_i)_{0\leq i}\subset {\cal N\cal T}({\cal Q})$ 
be any full numbered
splitting sequence. Then we have $\emptyset\not=P{\cal V}(\tau_{i+1})
\subset P{\cal V}(\tau_i)$ and hence
$\bigcap_iP{\cal V}(\tau_i)$ is a non-empty 
compact subset of ${\cal P\cal M\cal L}$. If
$\bigcap_iP{\cal V}(\tau_i)$
consists of a unique point, and the support of this projective
measured lamination is of the same topological type as $\tau_0$, 
then 
we call $(\tau_i)$
\emph{uniquely ergodic}. 

\begin{definition}\label{uniquelyergodic}
The sequence $(x_i)\in \Omega$ is called \emph{uniquely ergodic} if 
some (and hence every) 
full numbered splitting sequence $(\tau_i)\subset {\cal N\cal T}({\cal Q})$ 
which realizes
$(x_i)$ is uniquely ergodic.
\end{definition}

If $(\tau_i)\subset {\cal N\cal T}({\cal Q})$ realizes a uniquely ergodic sequence
$(x_i)\in \Omega$ then
for every $i$, a transverse measure on 
$\tau_i$ defined
by a point $\zeta\in \bigcap_i {\cal V}(\tau_i)$ 
is positive on every branch of $\tau_i$. Moreover, by Lemma 5.1 of \cite{H09a}, the 
sequence $(\tau_i)$ is uniquely determined
by $\tau_0$ and $\zeta$.

Let 
\begin{equation}\label{uniquelyer}
{\cal U}_+\subset \Omega\end{equation}
be the set of all
uniquely ergodic sequences. We define a
function $\rho:{\cal U}_+\to \mathbb{R}$
as follows. For $(x_i)\in {\cal U}_+$ choose
a full numbered splitting 
sequence $(\tau_i)\subset {\cal N\cal T}({\cal Q})$ 
which realizes $(x_i)$. 
Let $\mu\in  \bigcap_{i\geq 0}{\cal V}(\tau_i)$ be
carried by each of the train tracks $\tau_i$ and note that $\mu$ is uniquely
determined by this requirement up to scaling. 
Define 
\begin{equation}\label{rho}
\rho(x_i)=-\log (\mu(\tau_1)/\mu(\tau_0)).\end{equation} 
By equivariance
under the action of the mapping class group,
the number $\rho(x_i)\in (0,\infty)$ only depends on the
sequence $(x_i)\in {\cal U}_+$. In other words, $\rho$ is 
a function defined on ${\cal U}_+$. 
We have

\begin{lemma}\label{bounded} 
The function $\rho$ maps ${\cal U}_+$ to $(0,p\log 2]$. 
\end{lemma}
\begin{proof}
Let $(x_i)\in {\cal U}_+$ and 
choose a full numbered splitting
sequence $(\tau_i)\subset {\cal N\cal T}({\cal Q})$ which
realizes $(x_i)$. 
Let $\mu\in {\cal V}^{\cal P}(\tau_0)\cap \bigcap_{i\geq 0}{\cal V}(\tau_i)$.  
Thus $\mu$, viewed as a measured geodesic lamination,
is carried by each of the
train tracks $\tau_i$, and it defines
a transverse measure
$\mu$ on $\tau_0$ of total weight one.

Let $e$ be a large branch of $\tau_0$ and let
$\tau^\prime$ be the large train track which is obtained
from $\tau_0$ by a single split at $e$ and which is
splittable to $\tau_1$ (compare \cite{H09a} for details). Let
$e^\prime$ be the branch in $\tau^\prime$ which is
the \emph{diagonal} of the split of $\tau_0$ at $e$. This
means that $e^\prime$ is the small 
branch in $\tau^\prime$ which is the image of $e$ under the
natural bijection $\Lambda$ of the branches of $\tau_0$ onto
the branches of $\tau^\prime$.
Let $\mu^\prime$ be the
transverse measure on $\tau^\prime$ defined by the
measured geodesic lamination $\mu$. 
There are two  branches $b,d$ in $\tau$ incident on 
the two endpoints of $e$ 
such that
 \begin{equation}\label{linear}
 \mu(e)=
\mu^\prime(e^\prime)+\mu^\prime(\Lambda(b))+\mu^\prime(\Lambda(d)).\end{equation}
Moreover, we have
$\mu(a)=\mu^\prime(\Lambda(a))$ for every branch $a\not=e$  
of $\tau_0$ and hence 
$\mu(\tau_0)=1=\mu^\prime(\tau^\prime)+\mu^\prime(\Lambda(b))+\mu^\prime(\Lambda(d)) \leq 
2\mu^\prime(\tau^\prime)$ and 
\[\mu^\prime(\tau^\prime)\in [1/2,1].\]
Since $\tau_1$ can be obtained from $\tau_0$ by at most $p$ splits,
this immediately
implies that $\rho$ is nonnegative and bounded from
above by $p\log 2$.

On the other hand, by the definition of ${\cal U}_+$, the 
measure $\mu^\prime$ is positive
on every branch of $\tau^\prime$. Then the equation (\ref{linear}) 
shows that $\rho>0$ on ${\cal U}_+$. 
\end{proof}



In the following lemma, we view ${\cal U}_+$ as a subspace of the 
shift space $(\Omega,\sigma)$ with its standard topology, generated by
the clopen cylinder sets $C_\ell(x_i)=\{(y_i)\mid y_i=x_i\text{ for }-\ell\leq i\leq \ell\}$.

\begin{lemma}\label{roofcont}
The function $\rho:{\cal U}_+\to \mathbb{R}$ is 
continuous and only depends on the future.
\end{lemma}
\begin{proof} By construction, if $x_i=y_i$ for all $i\geq 0$,
then we have $\rho(x_i)=\rho(y_i)$, that is,  
 $\rho$ only
depends on the future.

To show continuity of $\rho$ 
let $(x_i)\in {\cal U}_+$. 
By the definition of the topology on the shift
space $(\Omega,\sigma)$ with basis the cylinder sets, 
it suffices to show that for every $\epsilon >0$ 
there is some $j\geq 0$ (depending on $(x_i)$)
such that 
\[\vert \rho(y_i)- \rho(x_i)\vert \leq \epsilon\]
whenever $(y_i)\in {\cal U}_+$ is such that
$x_i=y_i$ for $0\leq i\leq j$.

For this let $(\tau_i)$ be a full numbered splitting
sequence which realizes $(x_i)$.
Then $(\tau_i)$ determines a measured geodesic
lamination 
\[\mu= {\cal V}^{\cal P}(\tau_1)\cap \bigcap_{i\geq 0} {\cal V}(\tau_i). \]
By definition, $\rho(x_i)=\log \mu(\tau_0)$ where 
$\mu$ is viewed as a weight function on $\tau_0$
via a carrying map $\tau_1\to \tau_0$. Note that here we normalize
$\mu$ at $\tau_1$ for ease of exposition.

Let as before $p>0$ be the number of branches of a train
track in ${\cal N\cal T}({\cal Q})$. 
The set ${\cal V}^{\cal P}(\tau_1)$
of all transverse measures on $\tau_1$ 
of total mass one can be identified with 
a compact convex subset of $\mathbb{R}^p$. 
The natural projection 
\[\pi:{\cal V}^{\cal P}(\tau_1)\to 
{\cal P\cal M\cal L}\] is a homeomorphism onto its
image with respect to the weak$^*$-topology on ${\cal P\cal M\cal L}$
\cite{PH92}. Since by equation (\ref{linear}) in 
the proof of Lemma \ref{bounded}, a carrying map $\tau_1\to \tau_0$ induces 
a linear and hence continuous map ${\cal V}(\tau_1)\to {\cal V}(\tau_0)$,
there is an open neighborhood $V\subset {\cal P\cal M\cal L}$
of $\pi(\mu)$ with the following property.
Every
$\nu\in {\cal V}^{\cal P}(\tau_1)$ with $\pi(\nu)\in V$
defines a transverse measure
on $\tau_0$ whose total weight is contained
in the interval
$(e^{\rho(x_i)-\epsilon},e^{\rho(x_i)+\epsilon})$. 

Now for every $j>0$, the set 
$P{\cal V}(\tau_j)$ 
of all projective measured
geodesic laminations which are carried by $\tau_j$ 
is a compact subset of ${\cal P\cal M\cal L}$ containing 
$\pi(\mu)$,
and we have $P{\cal V}(\tau_j)
\subset P{\cal V}(\tau_i)$ for $j\geq i$ 
and $\bigcap_jP{\cal V}(\tau_j)=\pi(\mu)$.
As a consequence, there is some $j_0>0$ such that
$P{\cal V}(\tau_{j_0})\subset V$. 
By the definition of $\rho$,
this implies that 
the value of $\rho$ on the intersection with ${\cal U}_+$ of the 
cylinder $\{(y_i)\mid y_j=x_j$ for
$0\leq j\leq j_0\}$ is contained in the interval
$(\rho(x_i)-\epsilon,\rho(x_i)+\epsilon)$.
This shows the lemma.
\end{proof}

\subsection{Normal sequences are uniquely ergodic}
\label{sec:normal} 

Our next goal is to obtain a better understanding of the set 
${\cal U}_+\subset \Omega$ of uniquely ergodic sequences. 
It follows 
from the definitions that ${\cal U}_+$ is a Borel subset of $\Omega$.

Call a biinfinite sequence $(x_j)\in \Omega$ \emph{normal}
if every finite admissible sequence occurs in $(x_j)$ infinitely often
in forward and backward direction.
The following is the main result of this subsection.

\begin{proposition}\label{normal} A normal sequence
$(x_i)\in \Omega$ is uniquely ergodic.
\end{proposition}

The proof of Proposition \ref{normal} relies on some technical concepts and 
results which will also be important tools in Section \ref{sec:measureofmax}.


We first define a norm-like quantity
which measures the difference between
two projective measured geodesic laminations 
$[\mu],[\nu]$ which are carried by a train track $\tau$ and 
define positive weight functions on $\tau$. Namely,
choose representatives $\mu,\nu\in {\cal V}(\tau)$ of $[\mu],[\nu]$ 
and put
\begin{align}\label{norm} 
(\mu \mid \nu)_{\tau}^0 &=\max\{\mu(b)/\nu(b),\nu(b)/\mu(b)\mid b\text{ is a branch of }\tau\}
\text{ and }\\
([\mu]\mid [\nu])_\tau &=\min\{ (\mu \mid a \nu)_{\tau}^0\mid a>0\}. \notag
\end{align}
Note that $([\mu]\mid [\nu])_\tau$ indeed
only depends on the projective classes of $\mu,\nu$. Moreover,
we have $([\mu]\mid [\nu])_\tau=
([\nu]\mid [\mu])_\tau\geq 1$ for all $[\mu],[\nu]$, 
with equality if and only if $[\mu]=[\nu]$.

The following simple observation compares the behavior of these norms under
splitting operations. 

\begin{lemma}\label{preserve}
Let $(\tau_i)_{0\leq i\leq n}\subset
{\cal N\cal T}({\cal Q})$ be any finite full splitting
sequence, let $a_0>0$ and suppose that $\mu,\nu\in {\cal V}(\tau_n)$  
fulfill $\mu(b)\leq a_0\nu(b)$ 
for every branch $b$ of $\tau_n$.
Then the transverse measures $\mu_0,\nu_0$ on $\tau_0$ 
defined by $\mu,\nu$ via a carrying map $\tau_n\to \tau_0$ satisfy
$\mu_0(e)\leq a_0\nu_0(e)$ for every branch 
$e$ of $\tau_0$. In particular, we have 
\begin{equation*}
([\mu] \mid [\nu])_{\tau_n}\geq ([\mu]\mid [\nu])_{\tau_0}.\end{equation*}
\end{lemma}
\begin{proof} The lemma follows immediately
from the fact that 
the natural map ${\cal V}(\tau_n)\to {\cal V}(\tau_0)$ induced by a carrying map
$\tau_n\to \tau_0$ 
is the restriction of a linear map, explicitly given in (\ref{linear}),
from the finite dimensional vector space of weight functions on 
the branches of $\tau_n$ to the vector space of
weight functions on the branches of $\tau_0$ which preserves
positivity.
\end{proof}


Using the assumptions and notations from
Sections \ref{sec:strata}-\ref{sec:asymbolic} and the above notions, we
can now formulate the main
technical tool towards the proof of Proposition \ref{normal}.

\begin{lemma}\label{fillmore} 
Let $\tau_0\in {\cal N\cal T}({\cal Q})$ and let $\zeta\in {\cal V}(\tau)$ 
be a uniquely ergodic geodesic lamination with support 
in ${\cal L\cal L}({\cal Q})$. 
Let $(\tau_i)\subset {\cal N\cal T}({\cal Q})$
be the full splitting sequence starting at $\tau_0$ with
$\bigcap_i{\cal V}(\tau_i)=(0,\infty)\zeta$. Then 
there is some $n>0$ 
with the following properties.
\begin{enumerate}
\item There exists a number $\beta >0$ such that
$\mu(b)/\mu(b^\prime)\geq \beta$ for every $\mu\in {\cal V}(\tau_0)$ 
which is carried by $\tau_n$ and all branches $b,b^\prime$ of $\tau_0$. 
\item 
There is a number $\delta>0$ with the following 
property.
Let $\mu,\nu\in {\cal V}(\tau_n)$ be positive transverse measures; then 
\[([\mu] \mid [\nu])_{\tau_0}^{-1}
\geq ([\mu] \mid [\nu])_{\tau_n}^{-1}(1-\delta) +\delta.\]
%
\end{enumerate}
\end{lemma}
\begin{proof} Let $\tau_0\in {\cal N\cal T}({\cal Q})$ 
and let $\zeta\in {\cal V}(\tau_0)$ 
be a uniquely ergodic 
measured geodesic lamination whose support is contained
in ${\cal L\cal L}({\cal Q})$. Then the $\zeta$-weight of every branch
of $\tau_0$ is positive.

Let $(\tau_i)$ be the
full (numbered) splitting sequence such that for every $i$, 
the train track $\tau_i$ carries $\zeta$. 
We first claim that there exists a number $\ell>0$ so that 
the image of any branch $z$ of $\tau_{\ell}$ 
under a carrying map $F:\tau_{\ell}\to \tau_0$ 
which maps switches to 
switches and branches to edge paths on $\tau_0$
is all of $\tau_0$.

To show the claim 
we use the fact that every half-leaf of $\zeta$ is dense in $\zeta$
and the unzipping 
procedure for train tracks introduced on p.134 of \cite{PH92}. This construction
can be described as follows. 
Choose an arbitrary switch $v$ of $\tau_{0}$. The two small
branches $a_1,a_2$ incident on $v$ are contained in the boundary of some complementary
region $C$ of $\tau_0$. There is a corresponding complementary region
$C_\zeta$ of $\zeta$, and there are two oriented forward 
asymptotic boundary half-leaves
$\ell_1,\ell_2$ of $C_\zeta$ which correspond to the two sides of the 
complementary component $C$ of $\tau_0$ determined by the two small
half-branches $a_1,a_2$ incident on $v$. 

The two boundary half-leaves $\ell_1,\ell_2$ are mapped by a carrying map 
$\zeta\to \tau_0$ to a train path (an edge path of class $C^1$) on $\tau_0$, given
by a chain of oriented branches $b_1,b_2,\cdots$ where $b_1$ is the branch incident on 
$v$ and large at $v$.
Cut $\tau_0$ successively open along the branches $b_i$ as follows.
Replace first $b_1$ by a graph consisting of two edges 
(two copies of $b_1$) which are attached at a common vertex, which is the switch of $\tau_0$ on which 
the forward endpoint of the oriented edge $b_1$ is incident. 
Connect
the two univalent vertices of this graph, which correspond to the switch $v$, 
to the small branches $a_i$ in such a way that
the resulting graph is a train track $\tau^\prime$ as described on pages 135--137 of \cite{PH92}. 
The train track $\tau^\prime$ is carried by $\tau_0$ and it carries $\zeta$,  
and for $i=1,2$ it contains a branch $a_1^\prime$
which is mapped by a carrying map $\tau^\prime\to \tau_0$ to the union of $a_i$ with $b_1$. 
The train track $\tau^\prime$ may be not generic, but one can nevertheless 
repeat this construction with $\tau^\prime$ and the branch $b_2$ etc. 

Since the two half-leaves $\ell_1,\ell_2$ are mapped
by a carrying map $\zeta\to \tau_0$ onto $\tau_0$, there is a number $\ell>0$ so that 
after cutting $\tau_{0}$ open 
in this way along the branches $b_1,\dots,b_\ell$, the resulting train track $\xi$ 
is carried  by $\tau_{0}$ and has the property that its two branches that correspond to $a_1,a_2$ 
are mapped by a carrying map $\xi\to \tau_0$ onto $\tau_0$. The train track
$\xi$ may not be generic but can be modified with a sequence of shifts to a generic train track.
Repeat this procedure successively with all branches of $\xi$ which are not mapped onto 
$\tau_0$. In finitely many steps, one obtains a train track $\eta$ which is 
carried by $\tau_0$ and so that a carrying map $\eta\to \tau_0$ maps every branch of $\eta$ onto 
$\tau_0$. 

If $\eta^\prime$ is obtained from $\eta$ by a split, then 
$\eta^\prime$ is carried by $\tau_0$ and every branch of $\eta^\prime$ is mapped
onto $\tau_0$ by a carrying map. 
By Corollary 2.4.3 of \cite{PH92}, the train tracks 
$\tau_0,\eta$ can be split to the same train track $\sigma$ which carries
$\zeta$. Uniqueness of splitting sequences of train tracks which carry $\zeta$
up to ordering of the splits, established in Lemma 5.1 of \cite{H09a}, 
yields that $\sigma$ is splittable to $\tau_{\ell}$ for some $\ell>0$. The train track
$\tau_{\ell}$ has the properties stated in the claim.

%
%

By rescaling, assume now that $\zeta\in {\cal V}^{\cal P}(\tau_\ell)$. 
For each $i>\ell$ let
\[A_i={\cal V}^{\cal P}(\tau_{\ell})\cap {\cal V}(\tau_i)\] 
be the set
of all normalized transverse measures on $\tau_{\ell}$
defined by measured geodesic laminations which are 
carried by $\tau_{i}$. Then
$A_{i+1}\subset A_i$ and moreover
$\bigcap_iA_i=\{\zeta\}$. Since the $\zeta$-weight of 
every branch of $\tau_{\ell}$ is positive, there is a number
$\kappa >0$ so that for sufficiently large $i$, 
say for all $i\geq n$, and 
for every $\nu\in A_n$ we have 
\begin{equation}\label{uniformpositive}
\min\{\nu(b)/\nu(b^\prime) \mid b,b^\prime 
\text{ are branches of }\tau_{\ell}\}\geq \kappa.\end{equation}
Together with Lemma \ref{preserve}, 
this yields the first part of the lemma.

Let again $F:\tau_\ell\to \tau_0$ be the carrying map which maps switches to 
switches. 
To show the second part of the lemma let $n\geq \ell$ and 
assume that $\mu,\nu\in {\cal V}(\tau_n)$ are positive and let
\[([\mu]\mid [\nu])_{\tau_{n}}=
\max\{\mu(b)/\nu(b),\nu(b)/\mu(b)\mid b \text{ is a branch of }\tau_{n}\}=a_0>1.\]
By Lemma \ref{preserve}, we have 
$([\mu]\mid [\nu])_{\tau_\ell} =a_1\leq a_0$. 
By the definition of $([\mu]\mid [\nu])_{\tau_{\ell}}$, 
up to  
exchanging $\mu$ and $\nu$, there exists a branch 
$e$ of $\tau_\ell$ with $\nu(e)/\mu(e)=a_1$.

 For a branch $z$ of $\tau_{\ell}$ and a branch 
 $u$ of $\tau_0$ let 
$\beta(z)(u)\geq 1$ be the number of times the trainpath $F(z)$
passes through $u$.
Let 
\[m_1=\max \{\beta(z)(u)\mid z,u\} \text{ and } 
m_0=\min \{\beta(z)(u)\mid z,u\}\geq 1.\]
Denote as before by $p>2$ the number of branches
of $\tau_0$. Then for every branch $z$ of $\tau_{\ell}$ and 
$u$ of $\tau_0$ we have 
\begin{equation}\label{weight5}
\beta(z)(u)\geq \frac{m_0}{ m_1p} \cdot \sum_{s}\beta(s)(u).\end{equation}

By linearity, for any branch $u$ of $\tau_0$ we have
\begin{align*}
\nu(u) &=\sum_{s\not=e} \nu(s)\beta(s)(u)+ \nu(e)\beta(e)(u) 
\geq a_1^{-1} \sum_{s\not=e} \mu(s)\beta(s)(u) +a_1\mu(e)\beta(e)(u) \\
& =
a_1^{-1}\mu(u)+ (a_1-a_1^{-1})\mu(e)\beta(e)(u)
.\end{align*}
Since by the estimate (\ref{uniformpositive}) it holds
$\mu(e)\geq \kappa \max\{\mu(z)\mid z\}$, together with the estimate
(\ref{weight5}) 
we conclude that
\[\mu(e)\beta(e)(u) \geq \frac{m_0\kappa}{m_1p} \mu(u). \]
Now for $\delta=m_0\kappa/m_1p$ we obtain 
\[\nu(u)\geq a_1^{-1}\mu(u)+ \delta (a_1-a_1^{-1})\mu(u) =
(a_1^{-1}+ \delta (a_1-a_1^{-1}))\mu(u).\]
Since $1\leq a_1\leq a_0$, the branch $u$ of $\tau_0$ was arbitrary and 
we can exchange the roles of $\mu,\nu$, 
this yields the desired
estimate. 
\end{proof}

\begin{definition}\label{weaktight}
Call a finite full splitting sequence $(\tau_i)_{0\leq i\leq n}$
\emph{weakly tight} if the train tracks $\tau_n\prec\tau_0$ 
have properties (1) and (2) stated in  
Lemma \ref{fillmore}. 
\end{definition}

It follows from the definitions and 
Lemma \ref{preserve} that if $(\tau_i)_{0\leq i\leq n}$ is weakly 
tight, then the same holds true for the sequence
$(\tau_i)_{0\leq i\leq n+1}$ where $\tau_{n+1}$ is obtained from 
$\tau_n$ by a full split. 

To take full advantage of this concept we isolate a simple lemma which 
will be used several times in the sequel. 

\begin{lemma}\label{recursion}
Let $\beta\in (0,\frac{1}{2})$, let $\delta\in (0,\beta)$ and put 
$\kappa=1-\delta\in (\frac{1}{2},1)$. Define recursively a sequence
$(a_n)$ by $a_0=\beta$ and $a_n=(1-\delta)a_{n-1}+\delta$; then 
$a_n\geq 1-\kappa^n$ for all $n\geq 0$. 
\end{lemma}
\begin{proof}
We proceed by induction on $n$. The case $n=0$ follows from the requirement
that $\delta \leq \beta$ and hence $\beta\geq 1-\kappa$.

By induction, assume that the required inequality holds true for $n-1\geq 0$.
Then
\[a_{n}=(1-\delta)a_{n-1}+\delta\geq \kappa (1-\kappa^{n-1})+(1-\kappa)=1-\kappa^n\]
completing the induction step. 
\end{proof}

\begin{proof}[Proof of Proposition \ref{normal}]
Choose a weakly tight sequence $(y_i)_{0\leq i\leq n}\subset {\cal E}({\cal Q})$ and let $\beta >0$
and $\delta>0$ be as in Definition \ref{weaktight}.  Assume without loss of generality
that $\delta <\beta^2$. 
Let $(x_i)\in \Omega$ be a normal sequence; then there are numbers 
$i_0< i_1<\cdots$ so that the following holds true. 
\begin{enumerate}
    \item $i_{j+1}>i_j+n$ for all $j\geq 0$.
    \item For each $j$ and each $\ell \leq n$, we have $x_{i_j+\ell}=y_\ell$. 
\end{enumerate}
Choose a number $j>0$ and 
let $\mu,\nu\in {\cal V}^{\cal P}(x_{i_j+n})$. 
By the first condition in the definition
of a weakly tight sequence, 
the measures on $\tau_{i_j}$ defined by $\mu,\nu$ are positive, and
we have
$([\mu],[\nu])_{\tau_{i_j}}^{-1}\geq \beta^2$. 

By the second property of a weakly tight sequence, for all $\ell <j$ we have 
\begin{equation}\label{iteration}
([\mu],[\nu])_{\tau_{i_{\ell}}}^{-1}\geq ([\mu],[\nu])^{-1}_{\tau_{i_{\ell +1}}}(1-\delta)+\delta. \end{equation}
Lemma \ref{recursion} now yields that if $\mu,\nu$ are carried by $\tau_{i_{\ell}+n}$ then 
$([\mu] \mid [\mu])_{\tau_{0}}^{-1}\geq 1-\kappa^\ell$. But this implies that 
$([\mu]\mid [\nu])_{\tau_0}=1$ for $[\mu],[\nu]\in \bigcap_{i\geq 0} P{\cal V}(\tau_i)$ and hence
$[\mu]=[\nu]$. As a consequence, $\bigcap_{i\geq 0} P{\cal V}(\tau_i)$ consists of a unique point $[\mu]$. 
The transverse measure defined by $[\mu]$ on any of the train tracks $\tau_i$ is positive
on every branch.

To complete the proof we have to verify that the support of $[\mu]$ is of the same
topological type as $\tau_0$. This follows from another application of 
Corollary 2.4.3 of \cite{PH92}. Namely, otherwise there exists
 a splitting sequence $(\eta_j)$ starting at $\eta_0=\tau_0$ 
 (here the transition from $\eta_j$ to $\eta_{j+1}$ is a split at a single large 
 branch)
 which consists of 
 train tracks carrying $[\mu]$ and such that this splitting sequence 
 contains a collision, that is, a split followed by the diagonal of the split. 
By uniqueness of splitting sequences
starting from $\tau_0$ which consist of train tracks carrying $[\mu]$
as established in Lemma 5.1 of \cite{H09a}, this violates the fact that 
the measure $[\mu]$ is positive on every branch of any of the train tracks 
$\tau_i$. Together this shows that indeed, we have
$(x_i)\in {\cal U}_+$ as claimed. 
\end{proof}


As in Section 2, for  
$\tau\in {\cal N\cal T}({\cal Q})$ let
${\cal V}^*(\tau)$ be the set of all measured
geodesic laminations carried by $\tau^*$ and  
denote by
$P{\cal V}^*(\tau)$ the projectivization of 
${\cal V}^*(\tau)$.
If $\tau^\prime\in {\cal N\cal T}({\cal Q})$ 
is obtained from
$\tau\in {\cal N\cal T}({\cal Q})$ 
by a single split at a large branch $e$
and if $C$ is the matrix which describes the
transformation ${\cal V}(\tau^\prime)\to
{\cal V}(\tau)$ then the dual 
transformation
${\cal V}^*(\tau)\to {\cal V}^*(\tau^\prime)$ is 
given by the transposed matrix $C^t$
(Section 3.4 of \cite{PH92}). Thus Proposition \ref{normal} also 
applies to time reversal and shows the following corollary.

\begin{corollary}\label{fill} Let $(x_i)\in \Omega$ be
normal and let $(\tau_i)\subset {\cal N\cal T}({\cal Q})$
be a full numbered splitting sequence
which realizes $(x_i)$;
then $\bigcap_{i<0}P{\cal V}^*(\tau_i)$
consists of a single uniquely ergodic
projective measured geodesic lamination whose support is
contained in ${\cal L\cal L}({\cal Q})$.
\end{corollary}

We call the sequence $(x_i)\in \Omega$
\emph{doubly uniquely ergodic} if $(x_i)$ is uniquely ergodic
as defined above and if moreover 
for one (and hence every) full numbered splitting
sequence $(\tau_i)\in {\cal N\cal T}({\cal Q})$ 
which realizes $(x_i)$,
the intersection
$\bigcap_{i<0}{\cal P\cal V}^*(\tau_i)$ consists of a unique 
projective measured geodesic lamination  
whose support is contained in ${\cal L\cal L}({\cal Q})$.
By Lemma \ref{normal} and Lemma \ref{fill}, 
every normal sequence is doubly uniquely ergodic and hence the 
Borel set 
\begin{equation}\label{doubleunique}
{\cal U}\subset\Omega\end{equation}
of all doubly uniquely ergodic 
sequences $(x_i)\in \Omega$ is dense.

\subsection{Mapping doubly uniquely ergodic sequences into ${\cal Q}$}
\label{sec:semiconj}

Let $(x_i)\in {\cal U}$ and let $(\tau_i)$ be a 
full numbered splitting sequence which realizes $(x_i)$. 
Then $(\tau_i)$ determines a pair
$(\mu,\nu)$ of measured geodesic laminations 
by the requirement that 
$\mu\in {\cal V}^{\cal P}(\tau_0)\cap
\bigcap_{i\geq 0} {\cal V}(\tau_i)$,
that $\nu\in \bigcap_{i\leq 0}{\cal V}^*(\tau_i)$ and 
that $\iota(\mu,\nu)=1$.  Since the supports of both $\mu,\nu$
are contained in ${\cal L\cal L}({\cal Q})$,  
by equivariance under the
action of the mapping class group, this implies that 
every sequence $(x_i)\in {\cal U}$ determines a 
quadratic differential 
\begin{equation}\label{xi}
\Xi(x_i)\in {\cal Q}\subset {\cal Q}(m_1,\dots,m_\ell;-m).
\end{equation} 
More specifically, this quadratic differential is 
contained in the subset 
\begin{equation}
{\cal U\cal Q}\subset {\cal Q}\end{equation}
of all area one
quadratic differentials in ${\cal Q}$ 
whose vertical and horizontal measured
geodesic laminations
are uniquely ergodic, with support in 
${\cal L\cal L}({\cal Q})$. 
Note that ${\cal U\cal Q}$ is a $\Phi^t$-invariant Borel subset of 
${\cal Q}$. Its preimage ${\cal U}\tilde {\cal Q}$ 
in $\tilde {\cal Q}(S)$ (or in $\tilde {\cal H}(S)$)
is a $\Phi^t$-invariant Borel subset of $\tilde {\cal Q}$.

Recall that the roof function $\rho$ is defined on 
the dense shift invariant Borel set ${\cal U}\subset \Omega$, and by 
Lemma \ref{roofcont}, it is 
continuous, positive and bounded from above
by $p\log 2$.
The \emph{suspension}
for the shift $\sigma$ on the invariant subspace
${\cal U}$ with
roof function $\rho$ is the
space
\[X=\{(x_i)\times [0,\rho(x_i)]\mid (x_i)\in {\cal U}\}/\sim\] 
where the equivalence relation $\sim$ identifies
the point $((x_i),\rho(x_i))$ with the point
$(\sigma(x_i),0)$. Note that $\sim$ is a closed equivalence
relation on ${\cal U}$ since the function $\rho$ is continuous.
There is a natural flow $\Theta^t$ on $X$
defined by $\Theta^t(x,s)=(\sigma^jx,\tilde s)$ (for $t\geq 0$)
where $j\geq 0$ is such that
$0\leq \tilde s=t+s -\sum_{i=0}^{j-1}\rho(\sigma^ix)< \rho(\sigma^{j}x)$.

A \emph{semi-conjugacy} of $(X,\Theta^t)$
into a flow space $(Y,\Phi^s)$ is a surjective continuous
map $\Psi:X\to Y$ such that $\Phi^t \Psi(x)=\Psi(\Theta^tx)$ for
all $x\in X$ and all $t\in \mathbb{R}$. 
We call a semi-conjugacy $\Psi$ \emph{finite-to-one}
if the number of preimages of any point is finite.

By construction, there is a natural extension
of the map $\Xi$ defined in equation (\ref{xi})
to the suspension flow $(X,\Theta^t)$, again denoted by $\Xi$,
which commutes with the flows $\Theta^t$ and $\Phi^t$. 
We call such a map a \emph{partial semi-conjugacy}. Its 
image is contained in the $\Phi^t$-invariant 
Borel subset ${\cal U\cal Q}$.

The goal of this subsection is to establish the following result.

\begin{proposition}\label{semi} 
The map 
$\Xi:(X,\Theta^t)\to ({\cal U\cal Q},\Phi^t)$ is 
a finite-to-one semi-conjugacy.  
\end{proposition}

In particular, the image of the map $\Xi$ equals 
the $\Phi^t$-invariant subset ${\cal U\cal Q}\subset {\cal Q}$.

The proof of this proposition uses a technical result 
which constructs the preimage of a point $q\in {\cal U\cal Q}$ under the map 
$\Xi$ and shows that it is finite. 
Before we can establish this technical tool
we have to invoke some results from \cite{H09b}.

The \emph{curve graph}
${\cal C}(S)$ of $S$ 
is the metric graph whose vertices are the essential
simple closed curves on $S$ and where two such vertices are
connected by an edge of length one if and only if
they can be realized disjointly. 
The mapping class
group acts on the curve graph as a group of simplicial isometries.
The curve graph is 
a hyperbolic geometric metric graph \cite{MM99} which can be used for 
navigation in train tracks and in Teichm\"uller space in the following way. 

Any area one quadratic differential $\tilde q\in \tilde {\cal Q}$ 
defines a singular flat metric on $S$ of area one. The injectivity radius 
of this metric is bounded from above by a universal constant.
Thus there exists a number $\chi_0>0$ (depending on $S$) so that any such 
metric admits a simple closed geodesic of $\tilde q$-length 
(that is, length with respect to the flat metric defined by $\tilde q$) at most $\chi_0$.

Fix a number $\chi>\chi_0$ and define a map $\Upsilon_{\tilde {\cal Q}}:
\tilde {\cal Q}\to {\cal C}(S)$ by associating to a differential 
$\tilde q$ a simple closed curve of $\tilde q$-length at most $\chi$.
The following lemma 
is well known and can explicitly be extracted from 
\cite{R14}. It implies that the map $\Upsilon_{\tilde {\cal Q}}$ 
coarsely does not depend on choices.

\begin{lemma}\label{rafi}
For every $\chi>\chi_0$ there exists a number $R(\chi)>0$ so that for any 
$\tilde q\in \tilde {\cal Q}$, 
the distance in 
${\cal C}(S)$ between any two simple closed curves of 
$\tilde q$-length at most $\chi$ does not exceed $R(\chi)$.
\end{lemma}
\begin{proof}
As the distance in the curve graph between two simple closed curves
$a,b$ on $S$
is bounded from above by the geometric intersection number $\iota(a,b)+1$ 
between $a,b$ \cite{MM99}, 
it suffices
to show the following. For any $\chi>0$ there exists a number $E(\chi)>0$, and for 
any 
singular flat metric on $S$ defined by an area one quadratic differential $\tilde q$,
there exists a simple closed curve $c$ on $S$ so that $\iota(a,c)\leq E(\chi)$ 
for every simple closed curve $a$ of $\tilde q$-length at most $\chi$.

Theorem 3.1 of \cite{R14} shows that for any $\tilde q$ there exists a subsurface
$Y$ of $S$ bounded by simple closed curves 
of uniformly bounded \emph{extremal length} so that the $\tilde q$-length of any essential arc 
in $Y$ with endpoints on $\partial Y$ or any simple closed curve in $Y$ 
is bounded from below by a constant not
depending on $\tilde q$. Thus the number of intersections with $\partial Y$ 
of any simple closed curve
of $\tilde q$-length at most $\chi$ is bounded from above by a 
constant only depending on $\chi$. This shows the lemma.
\end{proof}

\begin{remark}\label{flathyp}
The proof of Lemma \ref{rafi} shows more specifically that for any 
area one quadratic differential $\tilde q$, 
the distance in the curve graph of any simple closed curve on 
$S$ of flat length at most $\chi$ to 
a curve on $S$ of small extremal length 
for the conformal structure on $S$ defined by $\tilde q$ is 
uniformly bounded. Additionally, 
curves of small extremal length are of small length for the
hyperbolic metric defining the same conformal structure (see \cite{R14} for a discussion).
\end{remark}

Let ${\cal T}(S)$ be the Teichm\"uller space
of all complete finite volume hyperbolic metrics on $S$, equipped with 
the \emph{Teichm\"uller metric} $d_{\cal T}$. 
The mapping class group
${\rm Mod}(S)$ acts properly discontinuously and isometrically 
on $({\cal T}(S),d_{\cal T})$.
By the construction on p.251 of \cite{H09b}, there exists a 
coarsely well defined map 
\begin{equation}\label{Lambda}
\Lambda:{\cal N\cal T}({\cal Q})\to {\cal T}(S)\end{equation}
which is coarsely equivariant under the action of the mapping class group
${\rm Mod}(S)$. This means that the map $\Lambda$ depends on choices, but
there exists a number $q>0$ so that if $\Lambda^\prime$ is any other choice, then
for all $\tau\in {\cal N\cal T}({\cal Q})$ and 
all $g\in {\rm Mod}(S)$, it holds
$d_{\cal T}(\Lambda(\tau),\Lambda^\prime(\tau))\leq q$ and
$d_{\cal T}(\Lambda(g\tau),g(\Lambda(\tau))\leq q$.

A \emph{vertex cycle} for a
large train track $\tau\in {\cal N\cal T}({\cal Q})$ 
is a simple closed curve carried by $\tau$ whose counting
measure defines an extreme point for the space of all
transverse measures on $\tau$. The distance in 
${\cal C}(S)$ between any two vertex cycles of a train track on $S$ 
is bounded from above by a universal constant.

Recall the canonical projection $P:\tilde {\cal Q}(S)\to {\cal T}(S)$. The
following statement is Lemma 4.3 of \cite{H09b} which will be used 
as an essential tool for the proof of Proposition \ref{semi}. For its formulation
and later use, for $\tau\in {\cal N\cal T}({\cal Q})$ define
\begin{equation}\label{qp}
{\cal Q}^{\cal P}(\tau)
\end{equation}
to be the set of all differentials $\tilde q\in {\cal Q}(\tau)$ with the property
that the vertical measured geodesic lamination of $\tilde q$ is contained in 
${\cal V}^{\cal P}(\tau)$, that is, it deposits the total mass one on $\tau$.

\begin{lemma}[Lemma 4.3 of \cite{H09b}]\label{navigation}
There is a number $\ell >0$, and for every $\epsilon >0$ there is a number $m(\epsilon)>0$
with the following property. Let $\sigma,\tau\in {\cal N\cal T}({\cal Q})$ and assume that
$\sigma$ is carried by $\tau$ and that the distance in ${\cal C}(S)$ between a vertex
cycle of $\sigma$ and a vertex cycle of $\tau$ is at least $\ell$. 
Let $\tilde q\in Q^{\cal P}(\tau)$
be such that the vertical measured geodesic lamination $\tilde q^v$ of $\tilde q$
is carried by $\sigma$. If the total weight of the transverse measure on $\sigma$ defined by
$\tilde q$ is not smaller than $\epsilon$, then $d_{\cal T}(\Lambda(\tau),P\tilde q)\leq m(\epsilon)$. 
\end{lemma}

We are now ready to prove the main technical result of this subsection.
Recall from Lemma \ref{bounded} that the roof function $\rho$ is bounded from 
above by $p\log 2$.

\begin{lemma}\label{finite}
For $\tilde q\in {\cal U} \tilde{\cal Q}$
there is a neighborhood $V$ of $\tilde q$ in 
$\tilde{\cal Q}$, and there are finitely many
train tracks $\tau_1,\dots,\tau_n\in {\cal N\cal T}({\cal Q})$
(where $n\geq 1$ depends on $\tilde q$)
with the following property. If $\eta\in 
{\cal N\cal T}({\cal Q})$ is such that
$\Phi^t\tilde z\in {\cal Q}^{\cal P}(\eta)$ for some $\tilde z\in V$ and some 
$t\in [0,p\log 2]$ then $\eta\in \{\tau_1,\dots,\tau_n\}$.
\end{lemma}
\begin{proof} 
By coarse equivariance under the action of the mapping class group
and proper discontinuity of this action on ${\cal T}(S)$, 
for every $R>0$ the number of all elements 
$\eta\in {\cal N\cal T}({\cal Q})$ so that $d_{\cal T}(\Lambda(\eta),P\tilde q)\leq R$
is finite where $\Lambda:{\cal N\cal T}({\cal Q})\to
{\cal T}(S)$ is as in equation (\ref{Lambda}).
Thus 
for the proof of the lemma
it suffices to show that for $\tilde q\in {\cal U}\tilde  {\cal Q}$
there is a neighborhood $V$
of $\tilde q$ in $\tilde {\cal Q}$ and a number $R>0$ with the following
property. If $\tilde z\in V$, if $t\in [0,p\log 2]$ and 
$\eta\in {\cal N\cal T}({\cal Q})$ are such that $\Phi^t\tilde z\in {\cal Q}^{\cal P}(\eta)$
then $d_{\cal T}(\Lambda(\eta),P\tilde q)\leq R$. 

Let $\tilde q\in \tilde {\cal Q}$ be a differential with the property that the
vertical measured geodesic lamination $\tilde q^v$ of $\tilde q$ is uniquely ergodic, with
support in ${\cal L\cal L}({\cal Q})$. 

By \cite{MM99} and Remark \ref{flathyp}, 
there exists a number $d>0$ not depending on 
$\gamma$ so that the
map $t\to \Upsilon_{\tilde {\cal Q}}(\Phi^t\tilde q)$ is an \emph{unparameterized $d$-quasi-geodesic} 
in ${\cal C}(S)$ (see Theorem 2.3 of 
\cite{H10b} for an explicit statement). 
This means that there exists a (perhaps finite) interval $(a,b)\subset \mathbb{R}$ and an 
increasing homeomorphism $\psi:(a,b)\to \mathbb{R}$ so that the map 
$t\to \Upsilon_{\tilde {\cal Q}}(\Phi^t\tilde q)$ is a $d$-quasi-geodesic in ${\cal C}(S)$.
Here we subsume additive and multiplicative control constants in the single number $d$.

As the measured lamination $\tilde q^v$ is uniquely ergodic and fills $S$, 
it follows from Theorem 1.1 of \cite{H06a} that the unparameterized quasi-geodesic 
$\Upsilon_{\tilde {\cal Q}}(\Phi^t\tilde q)$ $(t\in [0,\infty))$ 
is of 
infinite length (and the support of $\tilde q^v$ is its boundary point in the boundary of the curve
graph, see also the argument on p.124 of \cite{MM99}). 
Since uniform
quasi-geodesics in a hyperbolic space do not backtrack (see Lemma 2.4 
of \cite{H10b} for a precise account in the situation at hand), since the
Teichm\"uller flow is continuous, and since nearby differentials define uniformly bi-Lipschitz
equivalent flat metrics, 
there is a neighborhood $V$ of $\tilde q$ in
$\tilde {\cal Q}$, and there 
is a number $T>0$ such that
\[d(\Upsilon_{\tilde {\cal Q}}(\Phi^t \tilde z),
\Upsilon_{\tilde {\cal Q}}(P\Phi^s \tilde z))
\geq \ell +2b+2p\log 2\] for all $t\geq T$, all $s\in [0,p\log 2]$
and for all $\tilde z\in V$ where $\ell >0$ is as in Lemma \ref{navigation}, 
$b>0$ is a constant which will be determined below and where 
as before, $p\log 2$ is an upper bound for the roof function $\rho$. 

To determine the constant $b$ consider for the moment 
an arbitrary train track 
$\eta\in {\cal N\cal T}({\cal Q})$
and a differential $\zeta\in {\cal Q}^{\cal P}(\eta)$. 
Let $\zeta^v$ be the vertical measured lamination of $\zeta$. 
Vertex cycles of $\eta$
are precisely the extreme points for 
${\cal V}(\eta)$ and are furthermore represented by an embedded
trainpath which passes through every branch at most twice (Lemma 2.2 of \cite{H06a}).
In particular, their number is bounded from above by a universal constant.
Moreover, there exists a number $\kappa >0$ not depending $\eta$ or $\eta$,
and there is a vertex cycle $\chi\subset \eta$ so that 
$\zeta^v(b)>\kappa$ for every branch $b\in \chi$.

Now the $\zeta$-length of a closed curve $\alpha$ on $S$ is bounded from above by
$2(\iota(\zeta^h,\alpha)+\iota (\zeta^v,\alpha))$ where as before, $\iota$ stands
for intersection number. To estimate these quantities note that 
it follows from the continuity of the 
intersection form, the fact that transverse measures on $\eta$ supported 
on weighted geodesic multicurves are dense and Corollary 2.3 of \cite{H06a} that
$\iota(\zeta^v,\chi)\leq 2$.
Moreover, by p.197 of \cite{PH92}, the horizontal 
measured foliation $\zeta^h$ defines a \emph{tangential measure}
on $\eta$ (up to some equivalence relation) so that 
$1=\iota(\zeta^v,\zeta^h)=\sum_b \zeta^v(b)\zeta^v(b)$. As $\zeta^v(b)\geq \kappa$ for 
all $b\in \chi$, it follows that $\iota(\zeta^v,\chi)\leq \frac{1}{\kappa}$. 
Together we conclude that 
the $\zeta$-length of $\chi$ is bounded from above by a universal constant $\chi >0$. 
Thus by Lemma \ref{rafi}, the distance in the curve graph between the vertex cycle
$\chi$ and $\Upsilon_{\tilde {\cal Q}}(\tilde q)$ is bounded from above
by a unviersal constant. 
We let $b>0$ be such an upper bound.

Let $\tilde z\in V$, let $s\in [0,p\log 2]$ and let
$\eta\in {\cal N\cal T}({\cal Q})$ be such that $\Phi^s\tilde z\in {\cal Q}^{\cal P}(\eta)$.
We claim that $d_{\cal T}(\tilde z, \Lambda(\eta))\leq m(\epsilon)$ where
$\epsilon >e^{-T- p\log 2}$ and $m(\epsilon)>0$ is as 
in Lemma \ref{navigation}. This claim completes
the proof of the lemma. 

To show the claim 
modify $\eta$ with a full splitting sequence to a train track 
$\sigma\in {\cal N\cal T}({\cal Q})$ so that $\sigma$ carries the vertical
measured geodesic lamination $\tilde z^v$ of $\tilde z$ and such that, more 
precisely, we have 
$\Phi^{t}\tilde z\in {\cal Q}^{\cal P}(\sigma)$ for some 
$t\in [T,T+p\log 2]$. Such a train track $\sigma$ exists by Lemma \ref{bounded}.  
By the choice of the constants $b>0, T>0$,  
the distance in ${\cal C}(S)$ between a vertex cycle
of $\eta$ and a vertex cycle of $\sigma$ is at least
$\ell$. Thus 
$\sigma,\eta$ satisfy the hypothesis in Lemma \ref{navigation}
with a number $\epsilon\geq e^{-T-p\log 2}$. This implies that
$d_{\cal T}(P\tilde z,\Lambda(\eta))\leq m(\epsilon)$ which completes
the proof of  the lemma.
\end{proof}

\begin{proof}[Proof of Proposition \ref{semi}] 
For every $\tau\in {\cal N\cal T}({\cal Q})$
and every $\tilde q\in {\cal Q}(\tau)$ whose vertical measured
geodesic lamination $\tilde q^v$ 
has support $\nu\in {\cal L\cal L}({\cal Q})$,
there is a \emph{unique} full numbered splitting sequence 
$(\tau_i)_{i>0}$ issuing from $\tau_0=\tau$ 
which consists of train tracks carrying $\tilde q^v$ 
(see the paragraph before Section 3 of \cite{H09a} for uniqueness). 
As by Lemma \ref{finite}, there are only finitely
many $\eta\in {\cal N\cal T}({\cal Q})$ which can arise as the 
starting point for such a splitting sequence and controlled weight, 
it follows that for every $x\in X$ 
the cardinality of 
$\Xi^{-1}(\Xi(x))$ is finite. By construction, the map $\Xi$ commutes with 
the suspension flow $\Theta^t$ defined by the roof function $\rho$ and 
the Teichm\"uller flow on ${\cal Q}$ and hence it is 
a semi-conjugacy of $X$ onto a $\Phi^t$-invariant subset of ${\cal U\cal Q}$ provided
it it is continuous.

Continuity of the map $\Xi$ follows from the arguments used in the proof of Lemma
\ref{roofcont}. Namely, as the roof function $\rho$ is continuous, the flow
$\Theta^t$ is continuous. 
Since $\Xi$ conjugates the flow $\Theta^t$ to the flow 
$\Phi^t$ and both flows are continuous,
if now suffices to show the following. Let $(x_i)\in {\cal U}$ and 
let $s\in (0,\rho(x_i))$; then for every neighborhood 
$V\subset \tilde {\cal Q}$  of $\tilde q=\Xi((x_i),s)$, there exists
a number $\epsilon >0$ and a cylinder set 
$C_i=\{(y_i)\mid y_i=x_i\text{ for }-m\leq i\leq m\}$ so that 
$\Xi(C\cap ({\cal D\cal U},(t-\epsilon,t+\epsilon))\subset V$.

To this end recall from the proof of Lemma \ref{roofcont} that 
$\bigcap_i \Xi(C_i)=\tilde q$. 
Since the cylinder sets $C_i$ are open and closed, 
for a compact neighborhood $K$ of $\tilde q$ 
contained in the interior of $V$ there exists some 
$i$ so that $\Xi(C_i,s)\subset K$. Choose $\epsilon >0$ so that
$\cup_{-\epsilon <t<\epsilon}\Phi^tK\subset V$ and note that 
by equivariance, we have
$\Xi(C_i,(s-\epsilon,s+\epsilon))\subset V$. Since a cylinder set
in $\Omega$ is open and closed, this completes the proof of continuity.

We are left with showing that 
$\Xi(X)$ is all of ${\cal U\cal Q}$.
For this
let $q\in {\cal U\cal Q}$
and let $\tilde q$ be a lift of $q$ to $\tilde {\cal Q}$.
By Lemma \ref{numberedtype2}, 
there is some $\tau\in {\cal N\cal T}({\cal Q})$ which is contained in 
${\cal E}({\cal Q})$ and such that
$\tilde q\in {\cal Q}(\tau)$. 
If $\tilde q^v,\tilde q^h$ are the 
vertical and horizontal measured geodesic laminations
of $\tilde q$, respectively, 
then there is a 
biinfinite full numbered splitting 
sequence $(\tau_i)\subset {\cal N\cal T}({\cal Q})$ 
issuing from $\tau$ such that 
the intersection $\bigcap_{i>0}P{\cal V}(\tau_i)$
consists of a unique point which is 
just the class of $\tilde q^v$, and that the
intersection $\bigcap_{i<0}P{\cal V}^*(\tau)$ 
consists of a unique point which is
the class of $\tilde q^h$. Since the suspension of the 
orbit of a point
in ${\cal U}$ under the shift is mapped to
a biinfinite flow line of the Teichm\"uller flow, 
this implies that $q\in \Xi(X)$. 
\end{proof}

\begin{remark}
The Teichm\"uller flow is not hyperbolic,
and the map $\Xi$ is not bounded-to-one.
\end{remark}

\subsection{Surjectivity of the semi-conjugacy on probability measures}\label{sec:suronmeas}

Since the roof function
$\rho$ on $\cal U\subset \Omega$ 
is continuous, uniformly bounded and positive, every $\sigma$-invariant 
Borel probability
measure $\nu$ on $\Omega$ which gives full mass to ${\cal U}$ 
induces an invariant measure $\tilde \nu$
for  the suspension flow $(X,\Theta^t)$ 
of total mass $\int \rho d\nu<\infty$. The measure
$\tilde \nu$ is defined by $d\tilde \nu=d\nu\times dt$ where
$dt$ is the Lebesgue measure on the flow lines of the
suspension flow. 

Since by Lemma \ref{finite}, 
every $q\in {\cal U\cal Q}$ has a neighborhood $V$ in ${\cal Q}$ so that 
the cardinality of the preimage under $\Xi$ of any point $q\in V$ in bounded from 
above by a constant not depending on $q$, we can push forward the 
measure $\tilde \nu$ with the semi-conjugacy $\Xi$ and obtain 
a nontrivial finite $\Phi^t$-invariant Borel measure on 
${\cal Q}$ which we may 
normalize to have total mass one. 
Thus if ${\cal M}_\sigma({\cal U})$ denotes the space
of all $\sigma$-invariant Borel probability measures on $\Omega$ 
which give full measure to ${\cal U}$ then 
$\Xi$ induces a map 
\[\Xi_*:{\cal M}_\sigma({\cal U}) \to {\cal M}_{\rm inv}({\cal Q})\]
where ${\cal M}_{\rm inv}({\cal Q})$ is the space 
of $\Phi^t$-invariant Borel
probability measures on ${\cal Q}$. 
We equip both spaces with the
weak$^*$-topology. 
We have

\begin{lemma}\label{continuous}
The map $\Xi_*$ is continuous.
\end{lemma}
\begin{proof}
Since $\Omega$ is a compact metrizable space,
the space of all Borel probability measures on $\Omega$ equipped with
the weak$^*$-topology is compact and metrizable.
Thus we only have to show that whenever $\mu_i\to \mu$
in ${\cal M}_\sigma({\cal U})$
then $\Xi_*(\mu_i)\to \Xi_*(\mu)$.

Now by Lemma \ref{bounded},
the function $\rho$ is continuous on ${\cal U}$,
bounded and positive
and hence if $\mu_i\to \mu$ in ${\cal M}_\sigma({\cal U})$
then $\int\rho d\mu_i\to \int \rho d\mu>0$.
In particular, we have $\tilde \mu_i(X)\to \tilde \mu(X)$ 
where $\tilde \mu_i,\tilde \mu$ are the finite Borel
measures on the suspension space 
$(X,\Theta^t)$ defined by the measures $\mu_i,\mu$.
Therefore it holds $\Xi_*(\mu_i)\to \Xi_*(\mu)$ if and only
if for every continuous function $f$ on ${\cal Q}$ 
with compact support we have $\int f\circ \Xi \,d\tilde\mu_i\to 
\int f\circ \Xi \,d\tilde\mu$. However, since
$\Xi$ is continuous this is immediate.
\end{proof}

The next result completes the proof of Theorem \ref{coding}
from the introduction. For its formulation, denote by
${\cal M}_\Theta(X)$ the space of $\Theta$-invariant Borel probability
measures on $X$. Denote again by $\Xi_*:{\cal M}_\Theta(X)\to 
{\cal M}_{\rm inv}({\cal Q})$ the natural map.

\begin{lemma}\label{dense}
The map 
$\Xi_*:{\cal M}_\Theta(X)\to {\cal M}_{\rm inv}({\cal Q})$ 
is surjective.
\end{lemma}
\begin{proof} It suffices to show that every
ergodic $\Phi^t$-invariant Borel probability
measure on ${\cal Q}$ 
is contained in the image of
$\Xi_*$. 

Thus 
let $\nu$ be an ergodic $\Phi^t$-invariant
Borel probability measure on ${\cal Q}$. 
By the Birkhoff ergodic
theorem, there is a density point $q\in {\cal Q}$ 
for $\nu$ such that
the Borel probability measures 
\[\nu_T=\frac{1}{T}\int_{0}^T \delta_{\Phi^t q}dt\]
converge weakly to $\nu$ as $T\to \infty$ 
where $\delta_x$ denotes the Dirac mass at $x$.
By Lemma \ref{minimalfull} and the Poincar\'e
recurrence theorem,
we have
$q\in {\cal U\cal Q}$. Hence by Corollary \ref{semi}, 
up to possibly replacing $q$ by $\Phi^tq$ for some
$t\in \mathbb{R}$ there is some $(x_i)\in {\cal U}$
with $\Xi(x_i)=q$. 

By Lemma \ref{finite}, 
there is a neighborhood $V$ of 
$q$ in ${\cal Q}$ 
such that the preimage of $V\cap {\cal U\cal Q}$ under the map $\Xi$
is a \emph{finite} union of Borel sets $W_j\subset X$
($j=1,\dots,n$) with the property that the restriction of 
$\Xi$ to each of the sets $W_j$ is injective. 
Namely, note first that as $q\in {\cal U\cal Q}$, 
a preimage 
$\tilde q\in \tilde {\cal Q}$ of $q$ can not be a fixed point for 
a non-trivial element of ${\rm Mod}(S)$. Thus there exists
a neighborhood $\tilde V$ of $\tilde q$ which is mapped homeomorphically
into ${\cal Q}$. By making $\tilde V$ smaller if necessary, 
we may assume that $\tilde V$ 
has the 
properties stated in Lemma \ref{finite}. Then the projection $V$ of
$\tilde V$ to ${\cal Q}$ has the finiteness property we are 
looking for.

As $\nu_T\vert V\to \nu\vert V$ weakly as $T\to \infty$ 
and as the map $\Xi$ is equivariant with respect to the
suspension flow and the Teichm\"uller flow, we conclude
that the restriction to $\cup_{j=1}^nW_j$ of the Borel
probability measures 
\[\tilde \nu_T=\frac{1}{T}\int_0^T \delta_{\Theta^t(x_i)}dt\]
converge weakly to a measure on $\cup_iW_i$ which
projects to the measure $\nu$ on $V$. 

Since $q$ was an
arbitrary density point for $\nu$, we deduce that
the measures $\tilde \nu_T$ converge weakly to a
$\Theta^t$-invariant 
Borel probability measure on $X$ whose image
under the map $\Xi_*$ equals $\nu$. 
To be more precise, choose a countable partition 
${\cal V}=\cup_jV_j$ of 
a measurable subset of ${\cal U\cal Q}$ of full $\nu$-mass 
so that each $V_j$ has the properties stated in the second
paragraph of this proof. These sets define a countable collection of 
sets $W_j^\ell$ so that for each $j$, the collection
$\{W_j^\ell\mid \ell\}$ is finite and its union equals the preimage of 
$V_j$ under $\Xi$. By the discussion in the previous paragraph,
the restriction of the probability measures $\tilde \nu_T$ to 
$\cup_{j,\ell}W_j^\ell$ converges weakly to a measure $\tilde \nu$ which 
projects to $\nu$. But then $\tilde \nu$ is a probability measure
and hence a weak limit of the measures $\tilde \nu_T$. The measure
$\tilde \nu$ is invariant under the flow $\Theta^t$ and is mapped by
$\Xi$ to $\nu$. This completes the proof that 
$\Xi_*$ is surjective.
\end{proof}

Consider again the shift space $(\Omega,\sigma)$. Let $f:\Omega\to \mathbb{R}$
be a H\"older continuous function. Then $f$ defines an 
\emph{equilibrium state} $\mu_f$ which is an ergodic mixing 
$\sigma$-invariant Borel probability 
measure on $\Omega$. It is defined to be the unique invariant probability measure 
which maximizes the quantity $h_\mu+\int f d\mu$ where $h_\mu$ denotes the entropy of 
the invariant measure $\mu$. 
The measure $\mu_f$ gives full mass to normal sequences 
and hence it gives full support to the domain of the map $\Xi$. 
As a consequence, the following holds true.

\begin{theorem}\label{bernoulli} 
Any Gibbs equilibrium state on $\Omega$ induces via the map $\Xi_*$ a 
$\Phi^t$-invariant mixing Borel probability measure on ${\cal Q}$. 
\end{theorem}

\section{The measure of maximal entropy}\label{sec:measureofmax}

In this section we use the subshift of finite type constructed in
Sections \ref{sec:strata}-\ref{sec:symbolicdyn}  to 
show that for every component ${\cal Q}$ of 
a stratum, the $\Phi^t$-invariant 
probability measure $\lambda$ in the
Lebesgue measure class is the unique measure of maximal
entropy. For strata of abelian differentials, this
was earlier shown by Bufetov and Gurevich \cite{BG07}.

The strategy is as follows. Let $q\in {\cal Q}$ 
be any birecurrent point which is contained 
in its own $\alpha$-and $\omega$-limit set for the flow $\Phi^t$.
By the Poincar\'e recurrence theorem, 
for every $\Phi^t$-invariant Borel 
probability measure $\mu$ the set of 
such points is of full $\mu$-mass. Starting from the symbolic 
system constructed in Section \ref{sec:asymbolic}, 
we construct a topological Markov
shift on a countable set ${\cal S}$ of symbols, given
by a transition matrix $A=(a_{ij})_{{\cal S}\times {\cal S}}$.
The phase space of this shift is the space 
\[\Sigma=\{(y_i)\in {\cal S}^{\mathbb{Z}}\mid
a_{y_iy_{i+1}}=1\text{ for all }i\}.\]
We find a positive roof function 
$\phi:\Sigma\to (0,\infty)$ of bounded variation 
and only depending on the future
such that the suspension of the 
shift $T:\Sigma\to \Sigma$ with roof function $\phi$ admits
a \emph{bounded-to-one}
semi-conjugacy into $({\cal Q},\Phi^t)$. Its image ${\cal D}$ 
is $\Phi^t$-invariant and contains
all points $z\in {\cal Q}$ which contain 
the fixed quadratic differential 
$q$ in their $\alpha$-and $\omega$-limit set.
Since the Lebesgue measure $\lambda$ on ${\cal Q}$ has full support and is
ergodic 
under the Teichm\"uller flow, $\lambda$- almost every orbit for $\Phi^t$
is dense in ${\cal Q}$. Thus the set
${\cal D}$ is of full Lebesgue
measure. 

We use this coding and a result of Sarig \cite{S99}
to show that the supremum of the
entropies of all $\Phi^t$-invariant Borel probability measures on ${\cal Q}$ 
which give full mass to ${\cal D}$ is the supremum of the
entropies of all such measures which are supported in some compact
invariant subset of ${\cal Q}$. That the supremum of the entropies of 
\emph{all} invariant probability measures on ${\cal Q}$ supported in 
a compact subset of ${\cal Q}$ equals the 
entropy $h$ of the Lebesgue measure $\lambda$ was established in 
\cite{H10a}. Thus $\lambda$ is a measure of maximal entropy for the 
Teichm\"uller flow on ${\cal Q}$. 
We then apply  
a result of Buzzi and Sarig \cite{BS03} to conclude that there exists 
at most
one such measure. Together this shows
Theorem \ref{entropymax} from the introduction.
The implementation of this strategy is carried out in three subsections.

\subsection{A shift with countably many symbols}\label{sec:countable}

Let $q\in {\cal Q}$ be any point which is contained in 
both the $\alpha$- and $\omega$-limit set of its orbit
under the flow $\Phi^t$. The goal of this subsection is
to construct a shift $(\Sigma,T)$ 
with countably many symbols, a suspension over $\Sigma$
and a partial semi-conjugacy of this suspension into $({\cal Q},\Phi^t)$ 
so that the image of this map is the subset of ${\cal U\cal Q}$ of differentials
which contain $q$ in their $\alpha$- and $\omega$ limit set.

Let $\tilde q\in
\tilde {\cal Q}$ be a lift of $q$.
By Lemma \ref{minimalfull}, 
the vertical and horizontal measured geodesic
lamination of $\tilde q$ 
is strongly uniquely ergodic, with support in 
$ {\cal L\cal L}({\cal Q} )$.
Let $p>0$ be as in Lemma \ref{bounded}. 
By Lemma \ref{numberedtype} and 
Lemma \ref{finite}, there is a number $\ell\geq 1$, and there 
are $\ell$ large numbered train tracks 
$\tau_1,\dots,\tau_\ell\in {\cal E}({\cal Q})$
such that $\Phi^t\tilde q\in {\cal Q}^{\cal P}(\eta)$ for some
$t\in [0,p\log 2]$ and some $\eta\in {\cal N\cal T}({\cal Q})$ 
which is contained in ${\cal E}({\cal Q})$
if and only if $\eta\in \{\tau_1,\dots,\tau_\ell\}$
(see (\ref{qp}) for notations). 

By Lemma \ref{fillmore},
there is a number $n >0$ such that
the following holds true.
Let $i\leq \ell$ and let $(\sigma^i_j)_{0\leq j\leq n}$
be a full numbered splitting sequence
of length $n$ issuing from $\sigma^i_0=\tau_i$ with
the property that $\sigma^i_n$ carries the support $\zeta$
of the vertical measured geodesic lamination of $\tilde q$.
Then the sequence $(\sigma^i_j)_{0\leq j\leq n}$ is weakly tight
as defined in Definition \ref{weaktight}.

Recall the definition of the transition matrix for the 
subshift of finite type $(\Omega,\sigma)$ which defines admissible sequences.
Define ${\cal S}$ to be the set of all
finite admissible sequences $(x_i)_{0\leq i\leq s}$ 
with
the following additional properties.
\begin{enumerate}
\item $s\geq 2n$ and the sequences $(x_j)_{0\leq j\leq n}$ and 
$(x_j)_{s-n\leq j\leq s}$ are realized
by one of the full splitting sequences 
$(\sigma^i_j)_{0\leq j\leq n}$ $(i\leq n)$.
\item There is no number
$t\in [n,s-n)$ such that the sequence $(x_j)_{t\leq j\leq t+n}$ is 
realized by one of the full splitting sequences $(\sigma_j^i)_{0\leq j\leq n}$.
\end{enumerate}
Note that ${\cal S}$ is a countable set.

Define a transition matrix
$A=(a_{ij})_{{\cal S}\times {\cal S}}$ by requiring
that $a_{ij}=1$ if and only if 
the sequence $(x_\ell)_{0\leq \ell \leq s}$ 
representing the symbol $i$ and the 
sequence $(y_t)_{0\leq t\leq u}$ representing the symbol $j$ 
satisfy
$y_t=x_{s-n +t}$ for every $t\in \{0,\dots,n\}$.
By construction and the properties of the 
set ${\cal E}({\cal Q})$ established in Lemma \ref{numberedtype},  
\begin{equation}\label{bip}
\text{there are } i_1,\dots,i_N\in {\cal S}\text{ such that
for every }  u \in {\cal S},\, \exists j,v\text{ with }
a_{i_j u}a_{u i_v}=1.
\end{equation}
In other words, the transition matrix has the
\emph{big images and preimages (BIP) property}
as defined in \cite{S03}.

Let $\Sigma$ be the set of all biinfinite sequences
$(y_i)\subset {\cal S}^{\mathbb{Z}}$ with $a_{y_iy_{i+1}}=1$ for all $i$,
equipped with the (biinfinite) shift $T:\Sigma\to \Sigma$.
There is a natural continuous
injective map 
\[G:\Sigma\to \Omega\] whose
image contains the set of all normal sequences.
Here $\Omega$ is as in Lemma \ref{shift}. This map is defined by associating
to a sequence $(y_i)\in \Sigma$ the sequence $(x_j)\in \Omega$ obtained 
by viewing a symbol $y_i$ as a finite admissible word and noticing that 
the words represented by $y_i$ and $y_{i+1}$ overlap in the sense that
the subword of $y_i$ consisting of the last $n$ letters coincides with 
the subword of $y_{i+1}$ consisting of the first $n$ letters. 
Thus we can combine these two words to an admissible word obtained by erasing 
the last $n$ letters of $y_i$ and concatenating the resulting word with 
the word $y_{i+1}$.

\begin{lemma}\label{zoomingin}
It holds 
\[G(\Sigma)\subset {\cal U}.\]
\end{lemma}
\begin{proof} Let $(x_i)\in \Omega$ be a sequence in the image of 
a sequence in $\Sigma$ under the map $G$. 
Then there are numbers $i_1<i_2<\cdots$ so that for each $j$, there exists
some $s_j\in {\cal S}$ so that $s_j=(x_\ell)_{i_j-n\leq \ell \leq i_{j+1}}$.

Let $(\tau_i)\subset {\cal N\cal T}({\cal Q})$ be a full splitting sequence which
realizes $(x_i)$. 
Let $[\mu],[\nu]\in P{\cal V}(\tau_0)\cap \bigcap_jP{\cal V}(\tau_j)$.
We have to show that $[\mu]=[\nu]$ and that the support of $[\mu]$ is contained in 
${\cal L\cal L}({\cal Q})$. By the first requirement in the definition of 
weak tightness, applied to the finite admissible sequence $(x_\ell)_{i_j\leq \ell \leq i_{j+1}}$, 
the measures $[\mu],[\nu]$ are positive on each $\tau_n$, and there exists a number 
$\beta>0$ not depending on $j$ so that 
\begin{equation}\label{baseestimate}
([\mu] \mid [\nu])_{\tau_{i_j}}^{-1}\geq \beta.\end{equation}

The second property in the definition of a weakly tight sequence shows that 
there exists a constant $\delta >0$ not depending on $j$ so that
\begin{equation}\label{inductionback}
([\mu] \mid [\nu])_{\tau_{i_j}}^{-1} \geq 
([\mu] \mid [\nu])_{\tau_{i_{j+1}}}^{-1}(1-\delta) +\delta\end{equation}
for all $j$. An application of Lemma \ref{preserve} and Lemma \ref{recursion}
and letting 
$j$ tend to infinity yields 
$([\mu] \mid [\nu])_{\tau_0}=1$ and hence $[\mu]=[\nu]$. 
The fact that the support of $[\mu]$ is contained in 
${\cal L\cal L}({\cal Q})$ is a consequence of positivity 
as explained at the end of the proof of Proposition \ref{normal}.

The same argument is also valid by reversing the time, equivalently for 
$\cap {\cal V}^*(\tau_i)$. This shows the lemma.
\end{proof}

By the BIP-property (\ref{bip})
and the discussion at the end of Section \ref{sec:asymbolic}, we may
assume that the topological Markov chain
$(\Sigma,T)$ is topologically mixing.

\subsection{A roof function of bounded variation} 

Recall that the roof function $\rho$ is defined on 
${\cal U}\subset \Omega$. 
By Lemma \ref{zoomingin}, we then can  
define a roof function $\phi$ on $\Sigma$ by associating
to an infinite sequence $(y_i)\in \Sigma$ 
with 
$y_0=(x_i)_{0\leq i\leq s}$
the value 
\[\phi(y_i)=\sum_{i=0}^{s-n-1}\rho(\sigma^i(G(y_i))).\]
Note that by the definition of the roof function $\rho$, 
the value of $\phi$ on $(y_i)$ is just the logarithm of 
the total mass that the transverse measure 
$\mu\in {\cal V}^{\cal P}(x_{s-n})\cap \bigcap_{\ell\geq 0} {\cal V}(x_\ell)$
deposits on $x_0$. 
By the first requirement in the definition of a weakly 
tight sequence, if $p>0$ is as before the
number of branches of a train track in ${\cal N\cal T}({\cal Q})$, then 
this mass is at least $p$. In other words, 
the function $\phi$ is bounded from below by a 
positive constant $\log p>0$, is
unbounded and only depends on the future.

For $m\geq 1$ define the $m$-th variation of $\phi$ by
\[{\rm var}_m(\phi)=\sup\{\phi(y)-\phi(z)\mid
y_i=z_i\text{ for }i=0,\dots,m-1\}.\]
The following is the main technical tool towards the proof of 
Theorem \ref{entropymax}.

\begin{lemma}\label{variation}
There are numbers $\theta\in (0,1)$ and 
$L>0$ such that
${\rm var}_m(\phi)\leq L\theta^m$ for all $m\geq 1$. In particular,
\[\sum_{m\geq 1}{\rm var}_m(\phi)<\infty.\] 
\end{lemma}
\begin{proof}
Let $m\geq 1$ and let $(y_i),(z_i)\in \Sigma$ be such that
$y_i=z_i$ for $i=0,\dots,m$.
By definition,
there is a finite full numbered splitting sequence
$(\tau_i)_{0\leq i\leq u}\subset {\cal E}({\cal Q})$, there
are numbers $n\leq \ell_1-n<\ell_1\leq \dots <\ell_{m}=u$,
and there are two uniquely ergodic 
measured geodesic laminations $\mu,\nu\in {\cal V}^{\cal P}(\tau_{\ell_1-n})$ 
with support in ${\cal L\cal L}({\cal Q})$ 
such that the following holds.
\begin{enumerate}
\item The measured geodesic lamination
$\mu,\nu$ projects to the vertical measured geodesic
lamination of $\Xi G(y_i),\Xi G(z_i)$, respectively
(where $\Xi$ is as in Section \ref{sec:symbolicdyn}). 
\item $\mu,\nu$ are carried by $\tau_u$.
\item The sequence $(\tau_i)_{0\leq i\leq n}$ and 
each of the sequences 
$(\tau_i)_{\ell_j-n\leq i\leq \ell_j}$ $(j\leq m)$
is weakly tight.
\item $\phi(y_i)=\sum_{j=0}^{\ell_1-n-1}\rho(\sigma^j G(y_i))$ and
$\phi(z_i)=\sum_{j=0}^{\ell_1-n-1}\rho(\sigma^j G(z_i))$.
\end{enumerate}

Let $\beta \in (0,1/2)$ and  
$\delta \in (0,\beta^2)$ be the control constants for the weakly tightness
property of the finite admissible sequences which were used for the
construction of the alphabet ${\cal S}$. We claim that for 
$\kappa=1-\delta
\in (1/2,1)$
we have 
\begin{equation}\label{inductionhyp}
1-\kappa^m\geq ([\mu]\mid [\nu])_{\tau_{\ell_1-n}}^{-1}.\end{equation}

For this recall from the definition of a weakly
tight sequence that for all $j$ 
we have $\mu(b)/\mu(b^\prime)\geq \beta $ for all branches $b,b^\prime$ of
$\tau_{\ell_j-n}$, and similarly for $\nu$.

Let $b_0$ be any branch of $\tau_{\ell_j-n}$ and let
$\hat \nu$ be the multiple of $\nu$ so that
$\hat \nu(b_0)=\mu(b_0)$. Then for all 
branches $b$ of $\tau_{\ell_j-n}$ we have 
\[\mu(b)/\hat \nu(b)\geq
\beta^2 \mu(b_0)/\hat \nu(b_0)=\beta^2\] and similarly
$\hat \nu(b)/\mu(b)\geq \beta^2$. 
This implies that indeed, 
$([\mu]\mid [\nu])_{\tau_{\ell_j-n}}^{-1}\geq \beta^2\geq 1-(1-\delta)$ for all $j$.
The claim now follows from Lemma \ref{recursion}, 
Lemma \ref{preserve} and 
Lemma \ref{fillmore}.

Let for the moment $\tau\in {\cal N\cal T}({\cal Q})$ be arbitrary and let 
$\zeta,\xi  \in {\cal V}^{\cal P}(\tau)$ be positive. Assume that 
$([\zeta]\mid [\xi])_{\tau} =a\geq 1$. Let $\hat \xi=c\xi\in {\cal V}(\tau)$
be such that
$\max \{\zeta(b)/\hat \xi(b),\hat \xi(b)/\zeta(b) \mid b\}=a$. Then we have 
\[a^{-1}\leq \hat \xi(\tau)\leq a \]
and hence 
\[\max\{\zeta(b)/\xi(b),\xi(b)/\zeta(b)\mid b\}\leq a^2.\]

An application of this simple estimate to normalized positive transverse measures
$\mu,\nu\in {\cal V}^{\cal P}(\tau_{\ell_1-n})$ 
on $\tau_{\ell_1-n}$ with $([\mu]\mid [\nu])_{\tau_{\ell_1-n}}=a\geq 1$ 
together with Lemma \ref{preserve} yields
$\mu(\tau_0)/\nu(\tau_0)\in [a^2,a^{-2}]$ and therefore 
if $\mu_0,\nu_0\in {\cal V}^{\cal P}(\tau_0)$ 
are the normalizations of $\mu,\nu$, then 
\[\vert \phi(\mu_0)-\phi(\nu_0)\vert =
\vert \log\mu(\tau_0)-\log\nu(\tau_0)\vert \leq -2\log a.\]

As a consequence, the estimate (\ref{inductionhyp}) 
implies that 
\[\vert \phi(y_i)-\phi(z_i)\vert
\leq -2\log(1-\kappa^{m})\text{ if }y_i=z_i\text{  for }0\leq i\leq m.\]
As $\frac{-\log(1-t)}{t}\to 1$ $(t\to 0)$, from this
the lemma follows.
\end{proof}

By Lemma \ref{zoomingin}, the function $\phi$ is defined
on the entire space $\Sigma$. This is equivalent to stating that
the image of $\Sigma$ under the map $G$ is contained in the set
${\cal U}$ on which the roof function $\rho$ is defined.  
In particular, we can construct the
suspension
$(Y,\Psi^t)$ over $\Sigma$ with roof function
$\phi$.

\begin{lemma}\label{semiconjugacy} 
There is a bounded-to-one partial semiconjugacy
$\Upsilon:(Y,\Psi^t)\to 
({\cal Q},\Phi^t)$. Its image contains
the set of all points $z$ whose $\alpha$- and 
$\omega$ limit set contains $q$.
\end{lemma}
\begin{proof} By construction of the roof function $\phi$ and Lemma \ref{zoomingin},
there is an obvious partial semi-conjugacy
$(Y,\Psi^t)\to (X,\Theta^t)$ whose image is contained in the domain of the 
semi-conjugacy $\Xi:(X,\Theta^t)\to ({\cal U\cal Q},\Phi^t)$
defined in Section \ref{sec:asymbolic}. 
The composition 
\[\Upsilon:(Y,\Psi^t)\to ({\cal Q},\Phi^t)\] 
of these maps is a partial semi-conjugacy. By construction, its image is the 
subset ${\cal D}$ of ${\cal U\cal Q}$ consisting of all points whose 
orbit contains $q$ in its $\alpha$- and $\omega$ limit set. 

We are left with showing
that $\Upsilon$ is bounded-to-one. 
Now note that by the construction of the alphabet ${\cal S}$, if 
$(y_i)\in \Sigma$ then the image of $(y_i)$ under the map 
$\Xi\circ G:\Sigma\to {\cal Q}$ is contained in a fixed 
small neighborhood $V$ of the differential $q$ used in the construction with 
the additional property that the cardinality of the preimage of 
$V$ in $\Omega$ is bounded from above by a fixed constant $k>0$.
But this yields that the map $\Upsilon$ is a most $k$-to-one.
\end{proof}

\subsection{Entropy computation} 

As before, let $h$ be the entropy of the $\Phi^t$-invariant 
Lebesgue measure on ${\cal Q}$. 
Let ${\cal M}_T(\Sigma)$ be the space of all $T$-invariant
Borel probability measures on $\Sigma$.  For 
$\mu\in {\cal M}_T(\Sigma)$ let 
\[{\rm pr}_\mu(-h\phi)=h_\mu-h\int \phi d\mu\]
where $h_\mu$ is the entropy of $\mu$. 
By \cite{S99}, under the assumptions at hand, the
\emph{Gurevich pressure} of the function $-h\phi$ is given by
\begin{equation}\label{gurevich}
{\rm pr}_G(-h\phi)=\sup\{{\rm pr}_\mu(-h\phi)\mid
 \mu\in {\cal M}_T(\Sigma),
{\rm pr}_\mu(-h\phi)\text{ is well-defined}\}.\end{equation}

The following observation relies 
on the results of Sarig \cite{S99}
and on \cite{H10a}.

\begin{lemma}\label{finitepressure}
${\rm pr}_G(-h\phi)\leq 0$.
\end{lemma}
\begin{proof} By Theorem 3 of \cite{H10a}, the entropy $h$ of the invariant 
Lebesgue measure $\lambda$ on ${\cal Q}$ 
equals the supremum of the topological entropies of 
all compact invariant subsets of ${\cal Q}$.

On the other hand, by Theorem 2 of \cite{S99},
${\rm pr}_G(-h\phi)$ equals the supremum of the 
quantity defined in (\ref{gurevich})
but restricted to invariant 
measures $\mu$ supported in 
compact subsets of $\Sigma$. 

By Abramov's formula, if $A\subset \Sigma$ is 
any compact invariant set and if $\mu$ is a $T$-invariant
Borel probability measure supported in $A$, then
the entropy of the induced invariant measure for the 
suspension flow $(Y,\Psi^t)$ equals
\[h_\mu/\int \phi d\mu.\] 
As a consequence, the Gurevich pressure of $-h\phi$ 
is nonpositive if the entropy of 
every $\Psi^t$-invariant Borel probability measure
on $Y$ which is supported in a compact set 
does not exceed $h$.

The partial semi-conjugacy $\Upsilon:(Y,\Psi^t)\to ({\cal Q},\Phi^t)$
is bounded-to-one, and 
it maps the suspension of a
$T$-invariant compact set $A\subset \Sigma$ to a compact
$\Phi^t$-invariant subset of ${\cal Q}$.
This implies that the entropy of a $\Psi^t$-invariant
Borel probability measure on $Y$ supported in a
compact set is bounded from above by the supremum of the
topological entropies of the restriction of the Teichm\"uller flow
to compact invariant subsets of ${\cal Q}$. 
This 
quantity equals $h$ by the first paragraph of this proof. 
The lemma follows.
\end{proof}

\begin{proof}[Proof of 
Theorem \ref{entropymax}]
Let $\mu$ be any $\Phi^t$-invariant 
ergodic Borel probability
measure on ${\cal Q}$ and let $q\in {\cal Q}$ be
a density point for $\mu$ which contains
itself in its $\alpha$- and $\omega$-limit set.
Use $q$ to construct the countable two-sided Markov shift 
$(\Sigma,T)$ with roof function $\phi$
and suspension flow $(Y,\Psi^t)$. 

By Theorem \ref{coding}, there exists a $\Theta^t$-invariant 
Borel probability measure $\tilde \mu$ on $X$ 
whose image under the map $\Xi_*$ equals $\mu$. 
Since for $\mu$-almost every $z\in {\cal Q}$, the orbit of 
$\Phi^t$ through $z$ contains $q$ in its 
$\alpha$ and $\omega$ limit set, Lemma \ref{semiconjugacy} shows that 
$\tilde \mu$ gives full mass
to $\Upsilon(Y)$. In particular, as $\Upsilon$ is bounded to one,
there exists a $\Psi^t$-invariant Borel probability measure $\hat \mu$ on 
$Y$ which is mapped by the partial semi-conjugacy $\Upsilon$ to $\mu$. 
By disintegration, the measure $\hat \mu$ induces a $T$-invariant measure on 
$(\Sigma,T)$.

Let $\Sigma^+$ be the set of all 
one-sided infinite admissible sequences
$(x_i)\subset {\cal S}^{\mathbb{N}}$ with 
$a_{x_ix_{i+1}}=1$ for all $i$, equipped with the
one-sided shift $T_+:\Sigma^+\to \Sigma^+$.
Since the roof function $\phi$ only depends on the future,
it defines a function on $\Sigma^+$ which will be
denoted again by $\phi$. Since the roof function $\phi$ is 
bounded from below by a positive constant, up to normalization, 
the measure $\hat \mu$ descends to a $T^+$-invariant
Borel probability measure on $\Sigma^+$.

Theorem 1.1 of \cite{BS03} states that given a 
topologically transitive countable Markov shift $(\Sigma,T)$
and a function $\phi:\Sigma \to [0,\infty)$ 
so that the Gurevich pressure of $-\phi$ is finite and of finite
variation, there exists at most one invariant probability measure
$\mu$ which maximes the quantity $h_\mu +\int \phi d\mu$.

By Lemma \ref{variation} and Lemma \ref{finitepressure}, 
the function $-h\phi$ on $\Sigma^+$ 
satisfies all the assumptions in this result. 
In particular, since 
${\rm pr}_G(-h\phi)\leq 0$, the
entropy $h_\mu$ of the invariant Borel
probability measure $\mu$ does not exceed $h$.
As a consequence, the Lebesgue 
measure 
$\lambda$ on ${\cal Q}$ 
is a measure of maximal entropy, and 
by Theorem 1.1 of \cite{BS03}, it
is unique with this property. In other words,
if $h_\mu=h$ then $\mu=\lambda$. 
Since $\mu$ was an arbitrary $\Phi^t$-invariant ergodic Borel
probability measure on ${\cal Q}$, 
this completes the proof of the Theorem.
\end{proof}

Any finite admissible sequence
$(x_0,\dots,x_{n-1})\subset {\cal S}$ 
defines the cylinder
\[[x_0,\dots,x_{n-1}]=\{(y_i)\in \Sigma^+\mid
y_i=x_i\text{ for }0\leq i\leq n-1\}.\]

A \emph{Gibbs measure} for the function
$-h\phi$ on $\Sigma^+$ is a Borel probability
measure $\nu$ with the following property.
There is a number $c>0$ so that for every
cylinder set $[x_0,\dots,x_\ell]$ and every
$z\in [x_0,\dots,x_\ell]$ 
we have
\[\nu[x_0,\dots,x_\ell]\in 
[c^{-1}e^{-h\sum_{0\leq i\leq \ell-1}\phi(T^iz)},
ce^{-h\sum_{0\leq i\leq \ell-1}\phi(T^iz)}].\]

As a consequence of
\cite{S03} and of the above discussion,
we conclude

\begin{corollary}\label{gibbs}
The Lebesgue measure is a Gibbs measure on $\Sigma^+$.
\end{corollary}


\bigskip

\noindent
MATHEMATISCHES INSTITUT DER UNIVERSIT\"AT BONN\\
ENDENICHER ALLEE 60\\ 
D-53115 BONN, GERMANY\\
\noindent
e-mail: ursula@math.uni-bonn.de

\end{document}